\documentclass[review,onefignum,onetabnum]{siamart171218}

\usepackage{lipsum}
\usepackage{amsfonts}
\usepackage{graphicx}
\usepackage{epstopdf}
\usepackage{algorithmic}
\usepackage{pgfplots}
\usepackage{graphicx}
\usepackage[caption=false]{subfig}
\usepackage{tikz,pgfplots}
\pgfkeys{/pgf/declare function={arccosh(\x) = ln(\x + sqrt(\x^2-1));}}
\nolinenumbers
\ifpdf
  \DeclareGraphicsExtensions{.eps,.pdf,.png,.jpg}
\else
  \DeclareGraphicsExtensions{.eps}
\fi

\usepackage{mathrsfs} 

\newsiamremark{remark}{Remark}
\newsiamremark{hypothesis}{Hypothesis}
\crefname{hypothesis}{Hypothesis}{Hypotheses}
\newsiamthm{claim}{Claim}

\headers{Synchronization in the complexified Kuramoto model}{Hsiao, T-Y., Lo, Y-F., and Wang, W.}

\title{Synchronization in the complexified Kuramoto model\thanks{\today
\funding{Ting-Yang Hsiao is supported by the ERC CONSOLIDATOR GRANT 2023 ``Generating Unstable Dynamics in dispersive Hamiltonian fluids'', Project Number: 101124921.
 Views and opinions expressed are however those of the authors only and do not necessarily reflect those of the European Union or the European Research Council. Neither the European Union nor the granting authority can be held responsible for them.}}}

\author{Hsiao, Ting-Yang\thanks{(Corresponding author) International School for Advanced Studies (SISSA), Via Bonomea 265, 34136, Trieste, Italy. 
 (\email{thsiao@sissa.it}),}
\and Lo, Yun-Feng \thanks{School of Electrical and Computer Engineering, Georgia Institute of
Technology, Atlanta, 30332, Georgia, USA 
 (\email{ylo49@gatech.edu}),}
\and Wang, Winnie \thanks{Department of Physics,University of
Wisconsin-Madison, Madison, 53706, Wisconsin, USA 
  (\email{wwang629@wisc.edu})}.}

\usepackage{amsopn}

\makeatletter
\newcommand*{\addFileDependency}[1]{
  \typeout{(#1)}
  \@addtofilelist{#1}
  \IfFileExists{#1}{}{\typeout{No file #1.}}
}
\makeatother

\newcommand*{\myexternaldocument}[1]{%
    \externaldocument{#1}%
    \addFileDependency{#1.tex}%
    \addFileDependency{#1.aux}%
}

\ifpdf
\hypersetup{
  pdftitle={An Example Article},
  pdfauthor={D. Doe, P. T. Frank, and J. E. Smith}
}
\fi

\myexternaldocument{ex_supplement}

\begin{document}

\maketitle

\begin{abstract}
 In this paper, we consider an $N$-oscillators complexified Kuramoto model. We first observe that there are solutions exhibiting finite-time blow-up behavior in all coupling regimes. When the coupling strength $\lambda>\lambda_c$, sufficient conditions for various types of synchronization are established for general $N \geq 2$. On the other hand, we analyze the case when the coupling strength is weak. For $N=2$ with coupling below $\lambda_c$, our complex-analytic approach not only recovers the periodic orbits reported by Th\"umler--Srinivas--Schr\"oder--Timme but also provides, for the first time, their exact period $T_{\omega,\lambda}=2\pi/\sqrt{\omega^{2}-\lambda^{2}}$, confirming full phase locking. Furthermore, for the critical case $\lambda = \lambda_c$, 
we find that the complexified Kuramoto system admits homoclinic orbits.  These phenomena significantly differentiate the complexified Kuramoto model from the real Kuramoto system, as synchronization never occurs when $\lambda<\lambda_c$ in the latter. For $N=3$, we demonstrate that if the natural frequencies are in arithmetic progression, non-trivial synchronization states can be achieved for certain initial conditions even when the coupling strength is weak. In particular, we characterize the critical coupling strength ($\lambda/\lambda_c = 0.85218915...$) such that a semistable equilibrium point in the real Kuramoto model bifurcates into a pair of stable and unstable equilibria, marking a new phenomenon in complexified Kuramoto models.
\end{abstract}

\begin{keywords}
Complexified Kuramoto model, Coupled oscillators, Synchronization, Kuramoto model, Lyapunov function, Cherry flow.
\end{keywords}

\begin{AMS}
  	34D06, 37C27, 37C20
\end{AMS}

\section{Introduction}\label{sec1}
Synchronization is the natural phenomenon of collective oscillation observed in systems consisting only of autonomous oscillators. This phenomenon has been widely studied since Huygens' Horologium Oscillatorium \cite{huygens1986christiaan}; for modern studies, we refer readers to the general description materials in \cite{reviewstrogatz2000, reviewstrogatz2001, bookpikovsky2001, reviewboccaletti2002, bookstrogatz2004, bookosipov2007}.

The first attempt at solving this problem came from Winfree \cite{reviewwinfree1966}, who studied the nonlinear dynamics of a coupled system with $N$ oscillators in the large-$N$ limit via a mean-field approach (see also \cite{reviewstrogatz2000}). Kuramoto refined this model in 1975 \cite{kuramoto1975self}. The resulting \emph{Kuramoto model} \cite{bookkuramoto1984} is a system of weakly coupled, nearly identical, and interacting limit cycle oscillators. 
For a system of $N$ oscillators (see, for instance, \cite{van1993lyapunov, jadbabaie2004stability, diego2005thermodynamic, lasko2007phase, bronski2011, rodrigues2016kuramoto, hsia2019synchronization, chen2020mathematical, ha2020asymptotic, bronski2021, chen2021synchronization, chen2022frequency, roberto2022, roberto2023}), the Kuramoto model is a system of $N$ nonlinear ordinary differential equations (ODEs): for $n=1,\ldots,N$,
\begin{align}\label{first}
{\dot{\theta}}_n = \omega_n + \frac{\lambda}{N} \sum^{N}_{m=1} \sin(\theta_m-\theta_n),
\end{align}
where for the $n$th oscillator, $\theta_n$ is its phase angle and $\omega_n$ its natural frequency. Here, both quantities are assumed to be real. $\lambda$, assumed positive, denotes the coupling strength between the $N$ oscillators.



Recently, researchers have shifted their focus to higher-dimensional interactions and generalized the Kuramoto models by using different algebraic structures. There are many ongoing efforts to further our understanding of the synchronization phenomena through the study of the \emph{Lohe model}, which is a non-Abelian generalization of the Kuramoto model \cite{lohe2009nonabelian, choi2014complete, ha2016collective, ha2018relaxation, ritchie2018synchronization, deville2019synchronization}. Concurrently, there have been ever-increasing interests in the complexified Kuramoto models \cite{ha2012class, thumler2023synchrony, lee2024complexified}; more generally, synchronization phenomena in a system of oscillators with complex-valued quantities \cite{ha2021collective, bottcher2023complex, cestnik2024integrability, hsiao2023synchronization, lee2024complexified, chen2024phase}. 

Ha et al. first proposed the complexified Kuramoto model \cite[Eq. (2.3)]{ha2012class} as an example of a general model \cite[Eq. (2.1)]{ha2012class} for flocking phenomena on an infinite cylinder. The authors also analyzed the behavior of non-identical oscillators when the coupling strength is large enough \cite[Assumption (H1)]{ha2012class} (the case $\lambda>\lambda_c$). Moreover, the authors assume that, initially, the spread in the real part ($x$) is strictly smaller than $\pi/2$, and that in the imaginary part ($y$) is bounded by some constant dependent on system parameter \cite[Assumption (H2)]{ha2012class}. 

Th\"{u}mler et al. considered the complexified Kuramoto model in the regime of weak coupling \cite{thumler2023synchrony}. The authors observed that for $N=2$ oscillators there exists a conserved quantity (``energy") when the coupling is weak, and analyzed the system behavior based on this quantity. They also conducted simulations for more than $N=2$ oscillators in the weak coupling regime \cite[Supplementary Material, Section VI]{thumler2023synchrony} and observed that solution behavior of the complexified model seems to imply some form of synchronization of the original (real-variable) Kuramoto model. 

A follow-up work \cite{lee2024complexified} to \cite{thumler2023synchrony} considered the complexified Kuramoto model in a new scenario where the coupling strength can take complex instead of real values, which all previous work assumed. The authors analyzed in depth the case of $N=2$ complexified Kuramoto oscillators for different cases of the complex coupling strength. They also studied systems of $N\geq 2$ complexified Kuramoto model and conducted numerical simulations of their behavior.

Although the complexified Kuramoto model has been studied in these previous works, and many distinctions of the dynamics of this model from that of the classical real Kuramoto model have been found, it appears to the authors that the possibility of solutions exhibiting finite-time blow-up behavior in the complexified model has never been observed in the literature. In Appendix~\ref{appen:finite-time_blow-up}, we explicitly show the finite-time blow-up behavior of one solution $(x_1(t),y_1(t))$, $(x_2(t),y_2(t))$ satisfying
\begin{align}
    | y_1(t) - y_2(t) |
    =
    \log \left( \frac{ \omega_1-\omega_2 }{\lambda} \csc( x_1(t) - x_2(t) ) \right)
\end{align}
to the complexifed model with $N=2$ oscillators and $\omega_1>\omega_2$. We note the existence of this particular finite-time blow-up solution for some initial conditions holds true \emph{regardless} of the relative strength of the coupling constant $\lambda$ to the natural frequency gap $\omega_1-\omega_2$. The existence of this finite-time blow-up solution marks a newly observed and key distinction between the complexified and the real Kuramoto model. Moreover, it highlights that the issue of global existence of solutions over time needs to be explicitly taken into account when analyzing stability around equilibria of these complexified models.

\subsection{Contributions}\label{subsec:contrib}

The contributions of this paper are listed below.

\begin{enumerate}
   \item 
   We further analyze the complexified Kuramoto model proposed in \cite{ha2012class}. Since this model can exhibit finite-time blow-up, we also derive sufficient conditions under which the system is globally well-posed (i.e., solutions exist for all time). In particular, based on various regimes of coupling strength to frequency gap ratio, we divide our analysis into different cases (cf. Definition \ref{def 5}), and apply different analytical tools for each case. We also define various modes of synchronization to make the meaning of each mode precise.
   \item 
    In the case $\lambda>\lambda_c$, many of our proofs analytically verify numerical observations in \cite{thumler2023synchrony}. For example, numerical observations in \cite[Section VI, Supplementary Material]{thumler2023synchrony} indicate correlations between real and complexified Kuramoto models; some of these correlations are analytically verified by Lemma~\ref{p and f sync for Im} and Theorem~\ref{F sync} of this paper.
\item 
In the “weak coupling” case $\lambda<\lambda_c$ with $N=2$ oscillators, we provide an alternative proof of the existence of periodic orbits around each equilibrium point by employing complex analysis. A conserved quantity was constructed in a recent work \cite{thumler2023synchrony} to characterize these periodic orbits. However, for more general complex parameters $\omega, \lambda \in \mathbb{C}$, a more natural and generalized framework is required (see, for example, \cite{lee2024complexified, thumler2025collective, lee2025hopf}). In Section~\ref{sec 4}, we adopt a multi-valued complex function approach and define a conserved quantity directly via its imaginary part. Moreover, by applying the residue theorem, we explicitly compute the period of the periodic orbit as $T_{\omega,\lambda} = 2\pi/\sqrt{\omega^2 - \lambda^2}$ (Theorem \ref{weak N=2}), which allows us to describe their properties more precisely. Also, we prove that when $\lambda=\lambda_c$ the system admits homoclinic orbits. While our analysis in this paper is restricted to the present case, the framework developed in Section~\ref{sec 4} is sufficiently general to be applicable to other regimes of parameters. We further make the interesting observation that full phase-locking synchronization is a necessary and sufficient condition for frequency synchronization in the real Kuramoto model \cite{hsiao2025equivalence}, but this equivalence no longer holds in the complexified model, with the aforementioned periodic orbits serving as counterexamples.  
    \item
    It is observed numerically and conjectured in \cite{thumler2023synchrony} that non-trivial phase-locking states exist when the number of complexified Kuramoto oscillators $N$ is greater than or equal to three. We are the first to analytically prove that non-trivial phase-locking state does exist when $N=3$, even when the coupling strength is weak ($\lambda<\Lambda_c$). Such analysis is generally difficult even for the real Kuramoto models, where most results such as \cite{reviewstrogatz2000,dorfler2011critical,bronski2011,ha2020asymptotic} rely on the coupling strength being strong. Notably, our analysis in the ``weak coupling" case (Section~\ref{4.2.2}) also validates some numerical observations in \cite[Fig. 2(b)]{maistrenko2004mechanism} for the Cherry flow in real Kuramoto systems. In particular, we get the exact threshold $\lambda/\omega = \sqrt{\frac{69-11\sqrt{33}}{8}} = 0.85218915...$ for the onset of Cherry flow \cite[Fig.2(b)]{maistrenko2004mechanism}.

\end{enumerate}

\subsection{Organization}\label{subsec:org}

We organize the paper in the following order. In Section \ref{sec 2.2}, we clarify different levels of synchronization for the complexified Kuramoto model in Definition \ref{def 2}-\ref{def 4} (see, for instance, \cite{dorfler2011critical, dorfler2012synchronization, dorfler2014synchronization}). We also introduce two critical coupling strengths, $\lambda_c$ and $\Lambda_c$, in Definition \ref{def 5}. 

In Section \ref{sec 3}, we consider the case when the coupling strength $\lambda>\lambda_c$ for $N \geq 2$ oscillators. Since the real ($x_n$) and imaginary ($y_n$) oscillations can interact with each other (via \eqref{eq x+iy}), we first demonstrate the oscillations of both parts individually by alternating the control between the real and imaginary parts. We estimate the difference of the real part between each pair of oscillators (Lemma~\ref{thm 2}). Then, we demonstrate that the oscillators arrive exponentially fast at both frequency and phase synchronization in the imaginary part (Lemma~\ref{p and f sync for Im}). Once we achieve this, we demonstrate that the ``whole" complexified Kuramoto system (i.e., $z(t)$) achieves frequency synchronization by using the Lyapunov energy function \eqref{H function} stated in Theorem \ref{F sync}. We also prove that, under the same assumptions, the system further achieves phase synchronization if and only if all natural frequencies coincide (Theorem \ref{3.5}).

In Section \ref{sec 4}, we recast two-oscillator complexified Kuramoto equations as a single complex ODE and construct—on the complex plane punctured at the zeros—a path-defined primitive (a quadrature). Along any trajectory, the real part of this primitive increases linearly in time while the imaginary part remains constant, yielding a natural conserved quantity. This construction applies to arbitrary complex parameters and unifies prior “energy”-type invariants by revealing their common analytic origin. The multivaluedness reflects branch/path choices but does not affect the conserved value along a fixed trajectory. Beyond this, using the residue theorem to evaluate the one-cycle contour integral of $(\omega-\lambda\sin z)^{-1}$, we obtain a closed-form periodic orbit with period $T_{\omega,\lambda}=2\pi/\sqrt{\omega^2-\lambda^2}$ (Theorem \ref{weak N=2}). This complements the quadrature construction by providing an explicit timing law for the cycles.

In Section \ref{sec 4.2}, we analyze the complexified Kuramoto system with $N=3$ oscillators under the ``Cherry flow" ansatz in \cite{maistrenko2004mechanism}. First, we consider the regime $\lambda<\Lambda_c$, in Section \ref{4.2.1}, where we find that in each $2\pi$-period there are two non-real equilibrium points, one a sink and the other a source. The existence of the sink equilibrium futher implies that non-trivial initial conditions exist (i.e., start close to the sink) such that the complexified Kuramoto system archives both full phase-locking and frequency synchronization in this super weak coupling regime. In particular, we develop a simple ``horizontal cutting lemma" (Lemma \ref{lemma 4.5}) to locate two equilibria and analyze their linear stability. 

Finally, when the coupling strength $\lambda$ satisfies $\Lambda_c \leq \lambda < \lambda_c$ (Sections \ref{4.2.2}, \ref{Lambda c lambda lambda c}), we show that there are full phase-locking states in the complexified system. In this paper, we show that as soon as the coupling strength $\lambda$ exceeds $\Lambda_c$ in the complexified Kuramoto system, the aforementioned real semistable equilibria can bifurcate into two equilibria where one is asymptotically stable while the other is unstable, thus demonstrating new synchronization phenomena in the complexified Kuramoto model versus the real one.

\section{The complexified Kuramoto model}\label{sec 2}

\subsection{Preliminaries} \label{sec 2.1}

In this paper, we consider a fully-connected network of $N$ coupled complexified Kuramoto oscillators with $z_n=x_n+iy_n\in\mathbb{C}$ denoting the angle of the $n$th oscillator and $\omega_i$ its natural frequency. We denote the coupling strength by positive real number $\lambda$. Following the complexified Kuramoto model proposed by Th{\"u}mler,  Srinivas, Schr{\"o}der, and Timme \cite{thumler2023synchrony}, we obtain
\begin{align} \label{main eq}
    \dot{z}_n=\omega_n+\frac{\lambda}{N}\sum_{m=1}^N \sin(z_m-z_n),
\end{align}
or equivalently, 
\begin{equation} \label{eq x+iy}
\left\{ \begin{aligned}
&\dot{x}_n=\omega_n+\dfrac{\lambda}{N}\sum_{m=1}^N \sin(x_m-x_n)\cosh(y_m-y_n),~\\ 
&\dot{y}_n=~~~~~~~~\dfrac{\lambda}{N}\sum_{m=1}^N \cos(x_m-x_n)\sinh(y_m-y_n),
\end{aligned} 
\right. 
\end{equation}
for all $n=1,2,\ldots,N$.

It is often useful to introduce a rotating frame by using the changes of the variables
\begin{align}
    z_n \mapsto z_n-t(\omega_1+\ldots+\omega_N)/N.
\end{align}
This observation may allow us to assume 
\begin{align} \label{zero mean omega}
\omega_1+\ldots+\omega_N=0,
\end{align}
without loss of generality. Throughout this paper, we assume that the sum of natural frequencies is zero.

\subsection{Definitions of synchronization} \label{sec 2.2}
In order to get a clearer understanding of physics described in \cite{thumler2023synchrony}, it is important to distinguish different synchronization concepts.

\begin{definition} [Full phase-locking synchronization] \label{def 2}
    A solution of the complexified Kuramoto model \eqref{main eq} achieves full phase-locking if for all $n,m=1,\ldots,N$,
    \begin{align}
        \limsup_{t\rightarrow\infty} |z_n(t)-z_m(t)|<\infty.
    \end{align}
\end{definition}
\begin{definition} [Frequency synchronization] \label{def 3}
    A solution of the  complexified Kuramoto model \eqref{main eq} achieves frequency synchronization if for all $n,m=1,\ldots,N$,
    \begin{align}
        \lim_{t\rightarrow \infty} |\dot{z}_n(t)-\dot{z}_m(t)|=0. 
    \end{align}
\end{definition}
\begin{definition} [Phase synchronization] \label{def 4}
    A solution of the  complexified Kuramoto model \eqref{main eq} achieves phase synchronization if for all $n,m=1,\ldots,N$,
    \begin{align}
        \lim_{t\rightarrow \infty} |z_n(t)-z_m(t)|=0. 
    \end{align}
\end{definition}
\begin{definition} [Critical coupling strength] \label{def 5}
    For the complexified Kuramoto model, two critical coupling strengths are defined as follows:
    \begin{align} \label{eq:lambda_c}
        \lambda_c:= \max_{n,m\in\{1,\ldots,N\}} |\omega_n-\omega_m|,
    \end{align}    
and
\begin{align} \label{eq:Lambda_c}
        \Lambda_c:= \inf_{\lambda}\{\lambda>0: \exists\, x\in\mathbb{R}^N\,\textrm{s.t.}\, \frac{\lambda}{N}\sum_{m=1}^N \sin(x_m-x_n)=-\omega_n
        ~\textrm{for all }n\}.
    \end{align}
\end{definition}

We pause to remark that Definition \ref{def 3} is equivalent to the following condition
\begin{align}
    \lim_{t\rightarrow \infty} |\dot{z}_n(t)|=0,
\end{align}
for all $n=1,\ldots,N$, since we assume the mean of natural frequencies is zero in \eqref{zero mean omega}\footnote{For $n=1,\ldots,N$, we have $ \lim_{t\to\infty} | \dot{z}_n(t) - \frac{1}{N} \sum_{m=1}^N \dot{z}_m(t) | \leq \lim_{t\to\infty} \frac{1}{N} \sum_{m=1}^N | \dot{z}_n(t) - \dot{z}_m(t) | = 0 $ if the solution achieves frequency synchronization. But $\sum_{m=1}^N \dot{z}_m(t) \equiv \sum_{m=1}^N \omega_m = 0.$}. Also, we want to emphasize that in the real Kuramoto model, one can prove that if its solution achieves full phase-locking synchronization, then it also achieves frequency synchronization\footnote{Denoting $H(t)=\sum_{m=1}^N \dot{\theta}_m^2(t)$, from \eqref{first}, we notice that $\int_0^t H(s) ds=\sum_{m=2}^N \omega_m (\theta_m(s)-\theta_1(s))|_0^t+\frac{\lambda}{N}\sum_{1\leq n<m\leq N}\cos(\theta_n(s)-\theta_m(s))|_0^t$. 
Also, we observe that $H(t)$ is uniformly continuous by showing $\dot{H}$ is bounded. Therefore, the solution achieving full phase-locking synchronization implies the boundedness of $|\theta_n(t)-\theta_m(t)|$ for all $n,m=1,...,N$ and all $t>0$, which in turn implies $\lim_{t\rightarrow \infty} H(t)=0$, or equivalently, the solution achieves frequency synchronization.}. However, this is not true in the complexified Kuramoto model, as we will demonstrate in the case for $N=2$ when the coupling strength is ``weak", i.e., $\lambda<\lambda_c=\Lambda_c$. For reference, in the real-valued Kuramoto model, the first critical coupling strength $\lambda_c$ coincides with the critical coupling $K_{\mathrm{critical}}$ defined in \cite[Theorem~4.1]{dorfler2011critical}. The second critical coupling strength $\Lambda_c$ corresponds to $k_c$ in \cite[Definition~2]{verwoerd2008global}, and is also consistent with $K^{(\mathrm{pl})}$ introduced in \cite{thumler2023synchrony}.

It is worth noting that $\lambda_c=\Lambda_c$ when $N=2$, whereas for $N\ge 3$, applying similar arguments in \cite{hsiao2025equivalence} one can show that $\lambda_c\ge \Lambda_c$. In addition, throughout the paper we sometimes refer to the regime $\lambda<\Lambda_c$ as the weak coupling strength, which is consistent with the terminology used in \cite{thumler2023synchrony}.

\section{Coupling strength $\lambda>\lambda_c$} \label{sec 3}
\begin{lemma} [Real part full phase-locking] \label{thm 2}
    Given $\delta\in(0,\pi)$, let the coupling strength $\lambda>\lambda_c/\sin(\delta)$. Let $(z_1(t),\ldots,z_N(t))$ be a (maximal) solution to the complexified Kuramoto model \eqref{main eq} with the initial condition $(x_1(0),\ldots,x_N(0))\in [0,\pi-\delta]^N$ and assume that it exists globally in forward time. Then the solution achieves full phase-locking in the real part; in particular, we have
    \begin{align}
        |x_n(t)-x_m(t)| < \pi - \delta ,
    \end{align}
    for all $n,m=1,\ldots,N$ and $t>0$.
\end{lemma}
\begin{proof}
    In the following, we divide the proof into two cases. 

In the first case, we assume $|x_n(0)-x_m(0)|<\pi-\delta$ for all $n,m=1,\ldots,N$. We claim that 
    \begin{align} \label{bdd of w}
        |x_n(t)-x_m(t)| < \pi-\delta,
    \end{align}
for all $n,m=1,\ldots,N$, for all $t \geq 0$. Suppose, on the contrary, that \eqref{bdd of w} does not hold. It means that there exist the first moment $t^*>0$ such that
    \begin{align} \label{=pi-delta}
        | x_n (t^*) - x_m (t^*) | = \pi-\delta,~\mbox{for}~\mbox{some}~n,m=1,\ldots,N,
    \end{align}
    and 
    \begin{align} \label{<pi-delta}
        \sup_{t\in[0,t^*-\epsilon]} |x_s(t)-x_l(t)|<\pi-\delta,
    \end{align}
    for all $s,l=1,\ldots,N$, for all $0<\epsilon \ll\ 1$.
    We may assume $x_n(t^*)-x_m(t^*)>0$ without loss of generality and notice that at this moment $x_m(t^*)=\min_j x_j(t^*)$ and $x_n(t^*)=\max_j x_j(t^*)$. Hence, by \eqref{=pi-delta} and \eqref{<pi-delta}, it is clear that $\dot{x}_n(t^*)-\dot{x}_m(t^*) \geq 0$. On the other hand, recalling \eqref{eq x+iy}, we see that
    \begin{align*} \notag
        \dot{x}_n(t)-\dot{x}_m(t)=&\;\omega_n-\omega_m+\dfrac{\lambda}{N} \sum_{l=1}^N \Bigg(\sin(x_l-x_n)\cosh(y_l-y_n)\\ \notag
        &~~~~~~~~~~~~~~~~~~~~~~~~~-\sin(x_l-x_m)\cosh(y_l-y_m)\Bigg).
    \end{align*}
    Therefore, through \eqref{=pi-delta}, we obtain that, at $t=t^*$,
\begin{equation} \label{wn-wm<0}
    \begin{aligned} 
        \dot{x}_n -\dot{x}_m =&\;\omega_n-\omega_m+\dfrac{\lambda}{N} \sum_{l=1}^N \Bigg(\sin(x_l-x_m-(\pi-\delta))\cosh(y_l-y_n)\\ 
        &~~~~~~~~~~~~~~~~~~~~~~~~~-\sin(x_l-x_m)\cosh(y_l-y_m)\Bigg)\\ 
        \overset{(a)}{\leq}&\;\lambda_c+\dfrac{\lambda}{N} \sum_{l=1}^N \Bigg(\sin(x_l-x_m-(\pi-\delta))-\sin(x_l-x_m)\Bigg)\\
        \overset{(b)}{\leq}&\;\lambda_c-\lambda \sin(\delta)\\ 
        <& \;0
        ,
\end{aligned}
\end{equation}
where $(a)$ is due to $0 \leq x_l - x_m \leq \pi-\delta$ (this follows from \eqref{=pi-delta}, i.e. $x_m=\min_j x_j$ and $x_n=\max_j x_j$) and $(b)$ is due to the inequality $\sin(x-a)-\sin(x) \leq -\sin(a)$ for $0\leq x \leq a < \pi$. This contradiction then verifies the claim.

For the second case, we assume the initial conditions satisfy $x_n(0)=\pi-\delta$ and $x_m(0)=0$ for some $n,m=1,\ldots,N$. Then the inequality \eqref{wn-wm<0} still holds when $t=0$. Hence, there is some $t_0>0$ such that $ 0 < x_n(t)-x_m(t) < x_n(0)-x_m(0) = \pi-\delta $ for all $t\in (0,t_0]$. (If there are multiple pairs of $(n,m)$ satisfying $x_n(0)=\pi-\delta$ and $x_m(0)=0$, we get $t_0^{(n,m)}>0$ for each pair. Then choose $t_0=\min_{(n,m)} t_0^{(n,m)}>0$.) Then, we can apply the argument in first case to the system \eqref{eq x+iy} with the initial condition at $t=t_0$.

Therefore, the proof is completed.
\end{proof}

Our next goal is to analyze the sufficient condition of phase and frequency synchronization for the imaginary part. In order to control the $y_n-y_m$, we need to further restrict the quantity for $x_n-x_m$ for all $n,m=1,\ldots,N$.

\begin{lemma} [Imaginary part phase and frequency synchronization] \label{p and f sync for Im}
   Suppose there exists a global solution $(z_1(t),\ldots,z_N(t))$ of the complexified Kuramoto model \eqref{main eq} defined on $[0,\infty)$. Suppose that for some $0<\delta_0<\pi/2$, \begin{align} \label{bdd condition for real part}
        \limsup_{t\rightarrow \infty} |x_n(t)-x_m(t)|<\frac{\pi}{2}-\delta_0,
    \end{align} for all $n,m=1,\ldots,N$. Then the imaginary parts synchronize in both phase and frequency, i.e.,
     \begin{align}
        \lim_{t\rightarrow \infty} |y_n(t)-y_m(t)|=0,
    \end{align}
    and
    \begin{align}
        \lim_{t\rightarrow \infty} |\dot{y}_n(t)-\dot{y}_m(t)|=0. 
    \end{align}
    Moreover, each convergence is exponentially fast.
\end{lemma}

\begin{proof}
    Let us define the phase difference function
    \begin{align} \label{YY}
        Y(t):=\max_{s,l\in\{1,...,N\}}|y_s(t)-y_l(t)|
    \end{align}
    for $t\geq 0$. It is obvious that $Y(t):\mathbb{R}^+\rightarrow \mathbb{R}$ is a continuous and piecewise smooth function. By \eqref{bdd condition for real part}, there exists a $T>0$ such that $|x_n(t)-x_m(t)|<\pi/2 - \delta_0$ for all $n,m=1,\ldots,N$ and $t>T$. Fix any $t>T$. There exists a pair $(n,m)\in\{1,\ldots,N\}^2$ such that $Y(t)=y_n(t)-y_m(t)$. Recalling \eqref{eq x+iy}, a straightforward calculation for this phase difference reveals that, at time $t$,
    \begin{align*}
        \dot{Y}&=\dot{y}_n-\dot{y}_m\\
        &=\dfrac{\lambda}{N}\sum_{l=1}^N \Bigg(\cos(x_l-x_n)\sinh(y_l-y_n)-\cos(x_l-x_m)\sinh(y_l-y_m)\Bigg)\\
        &\overset{(a)}{\leq} \dfrac{\lambda\cos(\frac{\pi}{2}-\delta_0)}{N}\sum_{l=1}^N \Bigg(\sinh(y_l-y_n)-\sinh(y_l-y_m)\Bigg)\\
        &\overset{(b)}{\leq} \dfrac{\lambda\sin(\delta_0)}{N}\sum_{l=1}^N \Bigg( (y_l-y_n)-(y_l-y_m)\Bigg)\\
        &\leq -\lambda \sin(\delta_0) (y_n-y_m).\\
        &=-\lambda\sin(\delta_0) Y
        ,
    \end{align*}
where $(a)$ is due to $y_m \leq y_l \leq y_n $ for all $l=1,\ldots,N$ and $|x_l-x_k|<\pi/2-\delta_0$ for all $l,k=1,\ldots,N$; $(b)$ is due to the mean-value theorem, $\sinh(b)-\sinh(a)=(b-a)\cosh(c)$ for some $c$ between $a$ and $b$, $y_m-y_n\leq 0$ and $\cosh(\cdot)\geq 1$. With this differential inequality in $Y$, thanks to Gr{\"o}nwall's inequality, we obtain
    \begin{align} \label{exponentially decreasing}
        Y(t)\leq Y(T) \exp(-\lambda\sin(\delta_0) (t-T)),
    \end{align}
    for any $t>T$. This means that the imaginary part achieves phase synchronization exponentially fast, so
    \begin{align}
        \lim_{t\rightarrow \infty} |y_n(t)-y_m(t)|=0,
    \end{align}
    and hence, frequency synchronization due to \eqref{eq x+iy}, and thus
    \begin{align} \label{dot y tends 0}
        \lim_{t\rightarrow \infty} |\dot{y}_n(t)-\dot{y}_m(t)|=0. 
    \end{align}
    This convergence is also exponentially fast since from \eqref{eq x+iy} we have $$ |\dot{y}_n(t)| \leq \lambda \sinh(Y(t)) $$ for $t>T$.
    This completes the proof.
\end{proof}

\begin{theorem} [Full phase-locking] \label{q p locking}
    Consider the coupling strength and initial conditions described in Lemma \ref{thm 2}. Let $(z_1(t),\ldots,z_N(t))$ be a global solution of the complexified Kuramoto model \eqref{main eq}, defined for all $t\geq 0$. Then the solution $(z_1(t),\ldots,z_N(t))$ achieves full phase-locking. 
\end{theorem}

\begin{proof}
By Lemma \ref{p and f sync for Im}, it is sufficient to show $\limsup\limits_{t\rightarrow \infty}|x_n(t)-x_m(t)|$ is bounded by $\pi/2$. In the following, we divide the proof into three cases. 

In the first case, we assume $\delta\in(\pi/2,\pi)$. By Lemma \ref{thm 2}, it is clear that 
\begin{align*}
        |x_n(t)-x_m(t)|< \pi - \delta < \pi/2,
    \end{align*}
for all $n,m=1,\ldots,N$ and $t>0$, so we are done.

For the second case, assume that $\delta\in(0,\pi/2)$. We first note that, if there exists \emph{any} time instant $t^*\geq 0$ such that $|x_n(t^*)-x_m(t^*)| \leq \delta$ for all $n,m=1,\ldots,N$, then applying Lemma \ref{thm 2} to the system re-started at time $t^*$ implies that $|x_n(t)-x_m(t)| \leq \delta < \pi/2$ for all $n,m=1,\ldots,N$ and $t\geq t^*$, thus we will be done.

It remains to show that there indeed exists such a time instant $t^*$. We prove this by contradiction: suppose that for all time $t\geq 0$, $x_\mathrm{max}(t) - x_\mathrm{min}(t) > \delta$, where we define $x_{\mathrm{max}}(t):=\max_{k=1,\ldots,N} x_k(t)$ and $x_{\mathrm{min}}(t):=\min_{k=1,\ldots,N} x_k(t)$.

Fix any time $t\geq 0$. For any $(n,m) \in \arg\max_{k=1,\ldots,N} x_k(t) \times \arg\min_{k=1,\ldots,N} x_k(t)$, we have
\begin{align} \label{eq 3.1415}
    \delta< x_n(t) - x_m(t) = x_{\mathrm{max}}(t)-x_{\mathrm{min}}(t) := \beta < \pi - \delta,
\end{align}
where the last inequality follows from Lemma \ref{thm 2}. For any such $(n,m)$, a straightforward calculation reveals that
\begin{equation} \label{case 2 argument}
\begin{aligned}
        \dot{x}_{n}(t)-\dot{x}_{m}(t)\leq&\;\lambda_c+\dfrac{\lambda}{N} \sum_{l=1}^N \Bigg(\sin(x_l-x_n)\cosh(y_l-y_n)\\
        &~~-\sin(x_l-x_m)\cosh(y_l-y_m)\Bigg)\\
    \overset{(a)}{\leq}&\;\lambda_c+\lambda\max_{x^*\in[\delta,\beta]} \Bigg(\sin(x^*-\beta)-\sin(x^*)\Bigg)\\
    \overset{(b)}{\leq}&\;\lambda_c-\lambda\sin(\delta)\\
    <\;&0.
\end{aligned}
\end{equation}
We pause to remark that $(a)$ is due to \eqref{eq 3.1415}; $(b)$ is due to the function $\sin(x-\beta)-\sin(x)$, defined on $0 \leq x \leq \beta < \pi-\delta$, achieves it maximum $ -\sin(\beta)$ at $x=0$ or $\beta$, and from \eqref{eq 3.1415} we have $\sin \beta > \sin \delta$, where $\beta$ is defined in \eqref{eq 3.1415}.

This indicates that if $x_\mathrm{max}(t)-x_{\mathrm{min}}(t)>\delta$, then it is decreasing with the rate of change at least $\lambda \sin(\delta)-\lambda_c$, which
is a fixed positive value. Thus we have
\begin{align}
    x_\mathrm{max}(t)-x_\mathrm{min}(t) \leq \delta,
\end{align}
for $t\geq T_c:=\frac{\pi-2\delta}{\lambda \sin(\delta)-\lambda_c}>0$, which is a contradiction.

Finally, for the third case, we assume $\delta=\pi/2$. Then there exists $\alpha\in(0,\pi/2)$ such that $\lambda>\lambda_c/\sin(\alpha)>\lambda_c$. Similarly to the second case, if there exists a time instant $t^*>0$ such that $|x_n(t^*)-x_m(t^*)| \leq \alpha $ for all $n,m=1,\ldots,N$, then applying Lemma \ref{thm 2} to the system re-started at time $t^*$ implies that $|x_n(t)-x_m(t)| \leq \alpha < \pi/2$ for all $n,m=1,\ldots,N$ and $t\geq t^*$, thus we will be done. Then we can prove by contradiction that there indeed exists such a time $t^*$, using the same contradiction argument as the second case, with $\alpha$ in place of $\delta$. 

The proof is now complete.
\end{proof}

\begin{theorem} [Frequency synchronization] \label{F sync}
     Consider the coupling strength and initial conditions described in Lemma \ref{thm 2}. Let $(z_1(t),\ldots,z_N(t))$ be a global solution of the complexified Kuramoto model \eqref{main eq}, defined for all $t \geq 0$. Then the solution achieves frequency synchronization.
\end{theorem}

\begin{proof}
    Based on Lemma \ref{p and f sync for Im} and Theorem \ref{q p locking}, the imaginary part of the solution achieves phase and frequency synchronization. It remains to show that the real part also achieves frequency synchronization. We consider an energy function as the following:
    \begin{align} \label{H function}
        \mathcal{H}(t):= \int_0^t \sum_{n=1}^N \dot{x}^2_n(s) \;ds.
    \end{align}
    We notice that $\mathcal{H}(t)$ is an increasing function with respect to $t$. The results proved in Lemma \ref{p and f sync for Im} and Theorem \ref{q p locking} imply that $\dot{x}_n(t)$ and $\ddot{x}_n(t)$ are bounded, for all $n=1,\ldots,N$. Hence, $\sum_{n=1}^N \dot{x}^2_n(t)$ is uniformly continuous. If $\mathcal{H}(t)$ is bounded by a time-independent constant, then 
    \begin{align}
        \lim\limits_{t\rightarrow \infty}\sum_{n=1}^N \dot{x}^2_n(t)=0.
    \end{align}
Recall the first equation of \eqref{eq x+iy} and we arrive at
\begin{align*}
    &\mathcal{H}(t) \\
   =&\int_0^t \left(\sum_{n=1}^N \omega_n \dot{x}_n(s) +\dfrac{\lambda}{N} \sum_{n,m=1}^N \sin(x_m(s)-x_n(s))\cosh(y_m(s)-y_n(s)) \dot{x}_n(s)\right)\;ds\\
   =&\sum_{n=1}^N \omega_n(x_n(t)-x_n(0))\\
    &-\dfrac{\lambda}{N}\sum_{1\leq m<n \leq N} \int_0^t \sin(x_n(s)-x_m(s))\cosh(y_m(s)-y_n(s))(\dot{x}_n(s)-\dot{x}_m(s))\; ds\\
    \leq&\left|\sum_{n=1}^N \omega_n(x_n(t)-x_n(0))\right|
    \\
    &+
    \dfrac{\lambda}{N}\left|\sum_{1\leq m<n \leq N}(\cos(x_n(s)-x_m(s))\cosh(y_m(s)-y_n(s))\Big|_{0}^{t}\right|\\
    &+\dfrac{\lambda}{N}\left|\sum_{1\leq m<n \leq N} \int_0^t \cos(x_n(s)-x_m(s))\sinh(y_m(s)-y_n(s))(\dot{y}_{m}(s)-\dot{y}_n(s)) \; ds\right|\\
    =:& ~\mathcal{I}(t)+\mathcal{II}(t)+\mathcal{III}(t).
\end{align*}
We want to demonstrate that $\limsup_{t\rightarrow \infty} (\mathcal{I}(t)+\mathcal{II}(t)+\mathcal{III}(t))$ is bounded by a constant. It is not difficult to directly observe that $\limsup_{t\rightarrow \infty}\mathcal{II}(t)$ is bounded since cosine is bounded by $1$ and $\limsup_{t\rightarrow \infty} \cosh(y_m(t)-y_n(t))=\cosh(0)=1$ due to Lemma \ref{p and f sync for Im} and Theorem \ref{q p locking}. Recalling the zero mean assumption for natural frequencies in \eqref{zero mean omega}, we can obtain the bounds for $\mathcal{I}(t)$ by
\begin{align*}
    \mathcal{I}(t)\leq \sum_{n=2}^N \left|\omega_n\right||\left|(x_n(t)-x_1(t))\right|+\sum_{n=2}^N \left|\omega_n\right|\left|(x_n(0)-x_1(0))\right|.
\end{align*}
By Lemma \ref{thm 2}, we obtain that $\limsup_{t\rightarrow \infty}\mathcal{I}(t)$ is bounded.

It remains to verify that $\limsup_{t\rightarrow \infty} \mathcal{III}(t)$ is bounded. Recalling Lemma \ref{p and f sync for Im} and Theorem \ref{q p locking} and \eqref{dot y tends 0}, there exists a $L>0$ and $T_0>0$ such that 
\begin{align} \label{v_mm dot is bdd}
   |\dot{y}_n(t)-\dot{y}_m(t)|<L,
\end{align}
for all $n,m=1,\ldots,N$ and $t>T_0$. Therefore, combining \eqref{v_mm dot is bdd}, Lemma \ref{p and f sync for Im} and Theorem \ref{q p locking}, we obtain
\begin{equation} \label{integration of sinh}
    \begin{aligned} 
    \limsup_{t\rightarrow \infty}\mathcal{III}(t)\leq & \dfrac{\lambda}{N} \sum_{1\leq m<n\leq N}\int_0^{T_0} \left|\sinh(y_m(s)-y_n(s))| |\dot{y}_{m}(s)-\dot{y}_n(s)\right|\; ds\\ 
    &+\dfrac{\lambda L}{N} \sum_{1\leq m<n\leq N}\int_{T_0}^\infty \sinh(\left|y_m(s)-y_n(s)\right|)\; ds.
\end{aligned}
\end{equation}

Due to Lemma \ref{p and f sync for Im}, we know that the imaginary part of the solution achieves phase synchronization exponentially fast. Therefore, this implies that the second integration \eqref{integration of sinh} is bounded, and hence $\limsup_{t \rightarrow \infty}\mathcal{III}(t)$ is bounded. 

To recapitulate, we have proven that $\mathcal{H}(t)$ is bounded by a constant which is independent of $t\in\mathbb{R}^+$, so the solution $(z_1(t),\ldots,z_N(t))$ of the complexified Kuramoto model \eqref{main eq} achieves frequency synchronization. This proof is now complete.
\end{proof}

\begin{theorem} [Phase synchronization if and only if identical natural frequencies] \label{3.5}
     Consider the coupling strength and initial conditions described in Lemma \ref{thm 2}.
    Let $(z_1(t),\ldots,z_N(t))$ be a global solution of the complexified Kuramoto model \eqref{main eq}, defined for all $t \geq 0$.
    Then the solution $(z_1(t),\ldots,z_N(t))$ of the complexified Kuramoto model  \eqref{main eq} achieves phase synchronization if and only if $\omega_1=\ldots=\omega_N$.
\end{theorem}
\begin{proof}
    First, we show that if $\omega_1=\ldots=\omega_N$, then the solution achieves phase synchronization. Recall that we assume $\sum_{n=1}^N \omega_n=0$; hence $\omega_1=\ldots=\omega_N=0$. By Lemma \ref{thm 2} and Lemma \ref{p and f sync for Im}, we know that the differences between imaginary parts, i.e., $\{y_n(t)-y_m(t)\}_{1\leq n,m\leq N}$, all tend to zero exponentially fast as $t\to\infty$ via \eqref{exponentially decreasing}. Also, proofs of Theorem \ref{q p locking} and Theorem \ref{F sync} imply that the differences between real parts, i.e., $\{x_n(t)-x_m(t)\}_{1\leq n,m\leq N}$, are bounded by $\pi/2$ and that the real parts achieve frequency synchronization. Considering all these and taking the limit $t\to\infty$ in the first equation of \eqref{eq x+iy}, we can show that the solutions achieve phase synchronization. 

    Second, we show that if the solution achieves phase synchronization, then the natural frequencies must coincide; since we assume $\sum_{n=1}^N \omega_n=0$, it suffices to show that $\omega_1=\ldots=\omega_N=0$. For any $n=1,\ldots,N$, taking $t\to\infty$ in the real-part equation in \eqref{main eq} yields $\omega_n=0$, using Theorem \ref{F sync} for the left-hand side, and the phase synchronization assumption on the right-hand side. 
    
    This proof is now complete.
\end{proof}

\begin{remark}
Throughout the above results we have assumed that the solution exists globally in forward time (that is, no finite-time blow-up occurs). 
However, for the complexified Kuramoto model the issue of finite-time blow-up is rather subtle; as discussed in the Introduction and Appendix~\ref{appen:finite-time_blow-up}, one can even construct explicit solutions with finite-time blow-up for certain initial data. 
In what follows we investigate sufficient conditions ensuring the global-in-time existence of solutions.
\end{remark}

\begin{theorem}
Let $\delta \in (\pi/2,\pi)$ and suppose $\lambda > \lambda_c/\sin(\delta)$. 
Consider the complexified Kuramoto model \eqref{main eq} with initial phases satisfying 
\[
|x_n(0)-x_m(0)| < \pi - \delta, \quad \text{for all } n,m=1,\ldots,N. 
\]
Then the corresponding solution $z(t)=(z_1(t),\ldots,z_N(t))$ exists globally in forward time and achieves both full phase-locking and frequency synchronization. In addition, the solution achieves phase synchronization if and only if the natural frequencies are identical, i.e.,
\(\omega_1=\cdots=\omega_N\).
\end{theorem} 

\begin{proof}
Define $\zeta_n := z_n - z_1$ and $\tilde{\omega}_n := \omega_n - \omega_1$ for $n=1,2,\ldots,N$. 
Then the system \eqref{main eq} can be rewritten as
\begin{align} \label{zeta}
\dot{\zeta}_n = \tilde{\omega}_n + \frac{\lambda}{N}\sum_{m=1}^N \Big(\sin(\zeta_m-\zeta_n)-\sin(\zeta_m)\Big), \quad n=2,\ldots,N.
\end{align}
By the local existence and uniqueness theorem, there exists $\tilde{T}>0$ such that the initial value problem \eqref{zeta} admits a unique solution $\zeta:=(\zeta_2,\ldots,\zeta_N)$ on $[0,\tilde{T}]$. 
Let
\begin{align} \label{T star}
T^* := \sup\{\, t>0 : \exists\, \zeta:[0,t)\to \mathbb{C}^{N-1} \text{ solving \eqref{zeta}}\,\}.
\end{align}

Suppose for contradiction that $T^*<\infty$. 
By the blow-up alternative, this would imply
$\lim_{t\uparrow T^*} |\zeta(t)| = \infty$. 
However, within the interval of existence $[0,T^*)$, one can apply the arguments from Lemma \ref{thm 2} and Lemma \ref{p and f sync for Im} to deduce uniform bounds:
\[
|\mathfrak{Re}(\zeta_n(t))| = |x_n(t)-x_1(t)| < \pi-\delta, 
\qquad
|\mathfrak{Im}(\zeta_n(t))| = |y_n(t)-y_1(t)| \le Y(0),
\]
for all $t<T^*$. 
This contradicts the blow-up alternative, hence $T^*=\infty$ and $\zeta$ exists globally. 
Since $z_1+\cdots+z_N\equiv C$ for some constant $C\in\mathbb{C}$, the same conclusion holds for $z_1$, and consequently for $z=(z_1,\ldots,z_N)$. 

Global existence being established, we can now invoke Theorem \ref{q p locking} and Theorem \ref{F sync} to conclude that $z$ achieves both full phase-locking and frequency synchronization. 
Finally, by Theorem \ref{3.5}, the solution achieves phase synchronization if and only if $\omega_1=\cdots=\omega_N$. 
\end{proof}

\begin{remark} 
    One can also study the orbital stability~\cite{choi2012asymptotic, ryoo2025asymptotic} of the complexified Kuramoto model in similar settings. We leave this as future work.
\end{remark}

\subsection{Numerical results}
\label{subsec:numerical}
In this subsection, we provide some numerical results to support our analytical conclusions. 

In Fig.~\ref{fig:strong}, we demonstrate the time evolution of $N=5$ oscillators in the complexified Kuramoto model, for the time duration $0\leq t\leq 50$. We randomly sample $N=5$ natural frequencies from the standard normal distribution, then center them so that $\sum_{n=1}^N \omega_N=0$, satisfying the assumption \eqref{zero mean omega}. We also scale them so that their range is $1$, thus $\lambda_c=1$ by Def. \eqref{def 5}. The realization of the natural frequencies for Fig.~\ref{fig:strong} is, listed as a vector, $\vec{\omega} := (\omega_1,\ldots,\omega_5) = (-0.14, -0.20, -0.32, -0.02, 0.68)$. Also, we choose the coupling strength $\lambda=1.1$ and the parameter $\delta$ for initial conditions as $\delta=\pi/2$ (see Lemma~\ref{thm 2}), so that $\lambda>\lambda_c/\sin(\delta)$ is satisfied, i.e., by definition we are in the \emph{strong coupling} regime. The initial conditions for the real and imaginary parts are sampled independently and uniformly at random from $[0,\pi-\delta]$. Note that for the premises of Lemma~\ref{thm 2} through \ref{p and f sync for Im} and Theorem~\ref{q p locking} through \ref{F sync} to hold, we only need the real parts to initially lie in $[0,\pi-\delta]$. The realization of the initial conditions in Fig.~\ref{fig:strong} are, for the real part, $\vec{x}(0) = (0.85, 0.36, 0.62, 1.10, 0.33)$, and for the imaginary part, $\vec{y}(0)=(1.18, 0.66, 0.39, 1.53, 1.30).$ 

In Fig.~\ref{fig:strong_omega_zero}, the time evolution of $N=5$ oscillators with identical natural frequencies in the complexified Kuramoto model, for the time duration $0\leq t\leq 50$. In particular, the natural frequencies $\vec{\omega} = (0,0,0,0,0)$ by the assumption \eqref{zero mean omega}. The other parameters and the initial conditions are chosen to be the same as those in Fig.~\ref{fig:strong}.

\begin{figure}[H]
    \centering
    \subfloat[Trajectories $\{(x_n(t),y_n(t))\}_{1\leq n\leq 5}$ for the time duration $0\leq t\leq 50$. Frequency synchronization in the imaginary part can be observed, consistent with Lemma \ref{p and f sync for Im}.\label{fig:strong_trajectory}]
        {\includegraphics[width=0.45\textwidth]{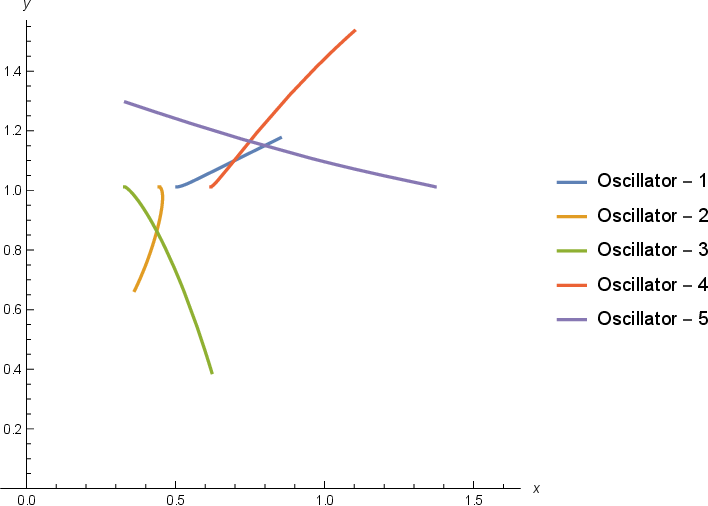}}
    \hfill
    \subfloat[$\{|x_n(t)-x_m(t)|\}_{1\leq m<n\leq N}$. Full phase-locking synchronization in the real part can be observed, consistent with Lemma \ref{thm 2}.\label{fig:strong_x}]
         {\includegraphics[width=0.45\textwidth]{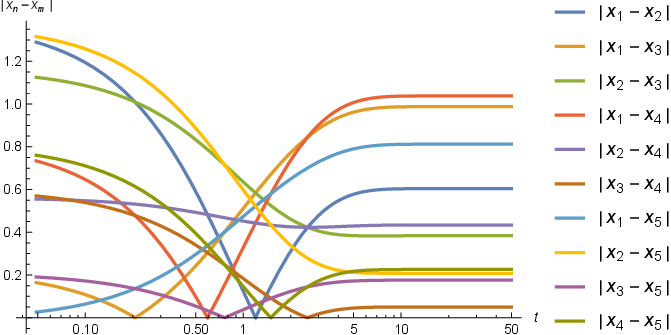}}
    \hfill
    \subfloat[$\{|y_n(t)-y_m(t)|\}_{1\leq m<n\leq N}$. Phase synchronization in the imaginary part can be observed, consistent with Lemma \ref{p and f sync for Im}. \label{fig:strong_y}]
         {\includegraphics[width=0.45\textwidth]{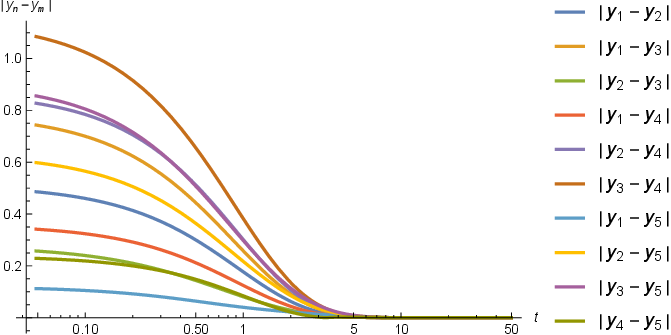}}
    \hfill
    \subfloat[$\{|\dot{y}_n(t)-\dot{y}_m(t)|\}_{1\leq m<n\leq N}$. Frequency synchronization in the imaginary part can be observed, consistent with Lemma \ref{p and f sync for Im}. \label{fig:strong_ydot}]
        {\includegraphics[width=0.45\textwidth]{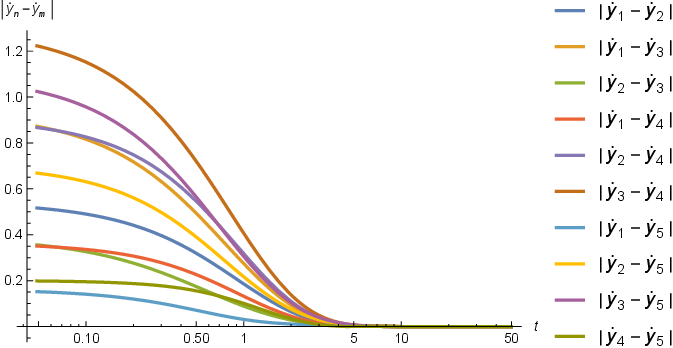}}
    \hfill
    \subfloat[$\{|z_n(t)-z_m(t)|\}_{1\leq m<n\leq N}$. Full phase-locking synchronization can be observed, consistent with Theorem \ref{q p locking}. \label{fig:strong_z}]
    {\includegraphics[width=0.45\textwidth]{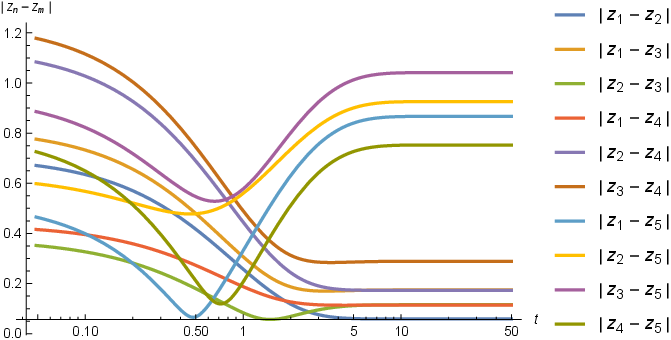}}
    \hfill
    \subfloat[$\{|\dot{z}_n(t)-\dot{z}_m(t)|\}_{1\leq m<n\leq N}$. Frequency synchronization can be observed, consistent with Theorem \ref{F sync}. \label{fig:strong_z_dot}]
    {\includegraphics[width=0.45\textwidth]
    {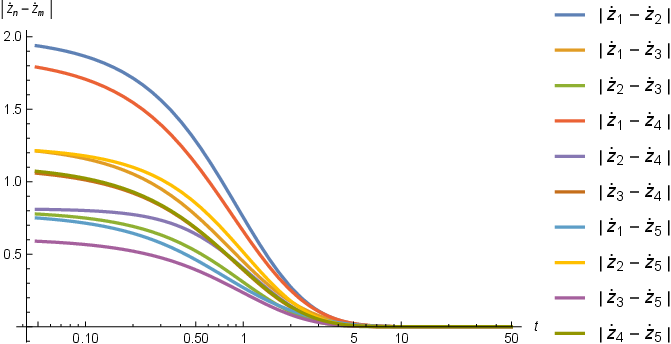}}

    \caption{Time evolution for $N=5$ oscillators under strong coupling. See Subsection~\ref{subsec:numerical} for parameters and initial conditions.}
    \label{fig:strong}
\end{figure}

\begin{figure}[H]
    \centering
    \subfloat[Trajectories $\{(x_n(t),y_n(t))\}_{1\leq n\leq 5}$ for the time duration $0\leq t\leq 50$. Phase synchronization can be observed, consistent with Theorem \ref{3.5} (and Theorem \ref{q p locking}). \label{fig:strong_trajectory_omega_zero}]
        {\includegraphics[width=0.45\textwidth]{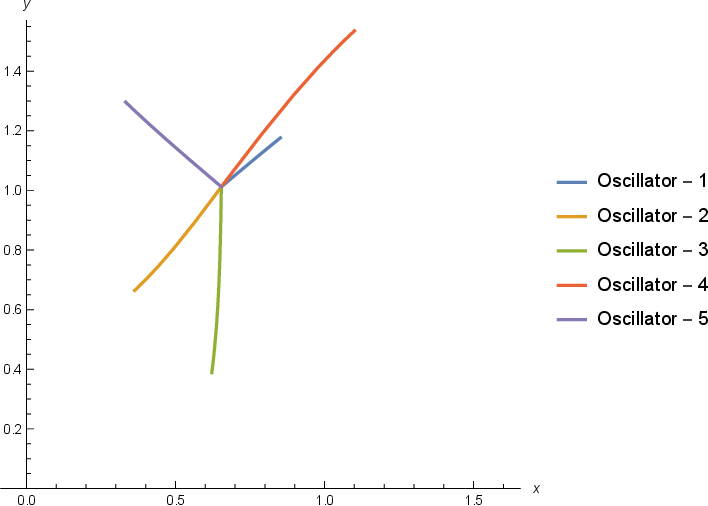}}
    \hfill
    \subfloat[$\{|x_n(t)-x_m(t)|\}_{1\leq m<n\leq N}$. Full phase-locking synchronization in the real part can be observed, consistent with Lemma \ref{thm 2}.\label{fig:strong_x_omega_zero}]
         {\includegraphics[width=0.45\textwidth]{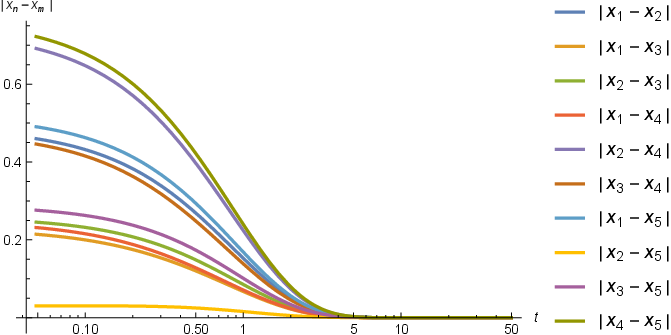}}
    \hfill
    \subfloat[$\{|y_n(t)-y_m(t)|\}_{1\leq m<n\leq N}$. Phase synchronization in the imaginary part can be observed, consistent with Lemma \ref{p and f sync for Im}. \label{fig:strong_y_omega_zero}]
         {\includegraphics[width=0.45\textwidth]{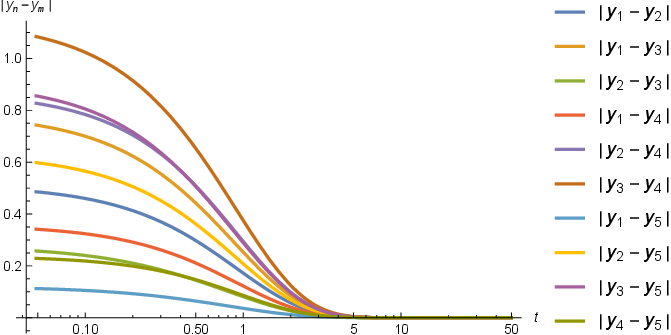}}
    \hfill
    \subfloat[$\{|\dot{y}_n(t)-\dot{y}_m(t)|\}_{1\leq m<n\leq N}$. Frequency synchronization in the imaginary part can be observed, consistent with Lemma \ref{p and f sync for Im}. \label{fig:strong_ydot_omega_zero}]
        {\includegraphics[width=0.45\textwidth]{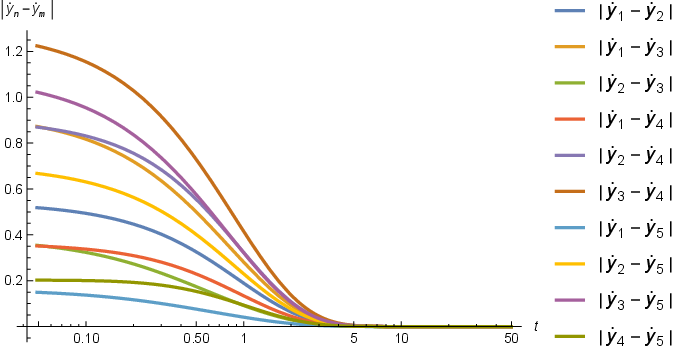}}
    \hfill
    \subfloat[$\{|z_n(t)-z_m(t)|\}_{1\leq m<n\leq N}$. Phase synchronization can be observed, consistent with Theorem \ref{3.5} (and Theorem \ref{q p locking}).\label{fig:strong_z_omega_zero}]
    {\includegraphics[width=0.45\textwidth]{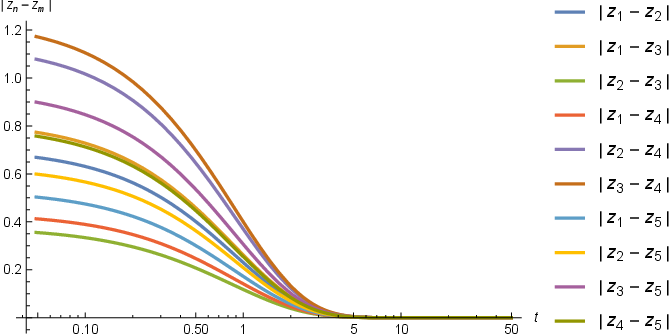}}
    \hfill
    \subfloat[$\{|\dot{z}_n(t)-\dot{z}_m(t)|\}_{1\leq m<n\leq N}$. Frequency synchronization can be observed, consistent with Theorem \ref{F sync}. \label{fig:strong_z_dot_omega_zero}]
    {\includegraphics[width=0.45\textwidth]
    {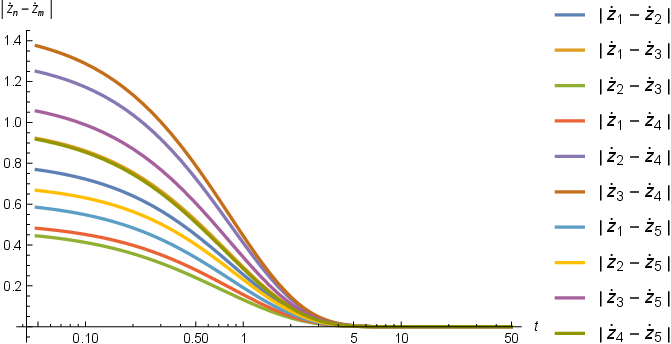}}
    
    \caption{Time evolution for $N=5$ oscillators with identical natural frequencies ($\boldsymbol{\omega}=\boldsymbol{0}$) under strong coupling. See Subsection~\ref{subsec:numerical} for parameters and initial conditions.}
    \label{fig:strong_omega_zero}
\end{figure}

\section{Two oscillators} \label{sec 4}

In this section, we analytically study synchronization in the complexified Kuramoto model with two oscillators. Let $z := z_1 - z_2 $ and $ \omega := \omega_1 - \omega_2 = 2\omega_1 $ in~\eqref{main eq} with $ N = 2 $. Then $ z$ satisfies the complex ODE  
\begin{align} 
    \label{n=2 complex}
    \dot{z} = f(z) := \omega - \lambda \sin(z).
\end{align}
The dynamics of~\eqref{n=2 complex} has been examined in~\cite{thumler2023synchrony} for real parameters and in~\cite{lee2024complexified} for complex ones. Both works identified a conserved ``energy'' function (cf.~\cite[Eq.~(5)]{thumler2023synchrony}, \cite[Eq.~(17)]{lee2024complexified}) within certain parameter regimes, though their forms appear different. Here we provide a unified construction of a conserved quantity valid for arbitrary \(\lambda,\omega \in \mathbb{C}\) with $\lambda\neq 0$. The method works, in fact, for any holomorphic function defined on the punctured complex plane—that is, on the complex plane with all points where the function vanishes removed. Inspired by
flow-box theorem for holomorphic vector fields, we have theorem as follows.

\begin{theorem} \label{local rectification thm}
Let $ \mathtt{f}: U \to \mathbb{C} $ be a holomorphic function, and suppose $ \mathtt{f}(z_0) \neq 0 $.
Then there exists a simply connected open neighborhood $ V \subset U $ of $ z_0 $
and a holomorphic function $ \mathscr{F}: V \to \mathbb{C} $ such that $ \mathscr{F}'(z) = 1/\mathtt{f}(z)$ for all $z\in V$.
Let $ z(t) $ be any (locally defined) solution of
\begin{align*}
\dot z = \mathtt{f}(z), \qquad z(t) \in V.
\end{align*}
Then
\begin{enumerate}
    \item \label{111}$ \dfrac{d}{dt} \mathscr{F}(z(t)) = 1 $, hence $ \mathscr{F}(z(t)) = t + C $.
    \item \label{222} The imaginary part $ I(z) := \mathfrak{Im} \mathscr{F}(z) $ satisfies
    \begin{align*}
        \frac{d}{dt} I(z(t)) = 0,
    \end{align*}
    i.e. $I(z)$ is constant along each trajectory.
    \item \label{333}Geometrically, the integral curves of the vector field \( \dot z = \mathtt{f}(z) \)
    in \( V \) are precisely the level curves
    \[
    \{ z \in V : \mathfrak{Im} \mathscr{F}(z) = \text{constant} \}.
    \]
\end{enumerate}
\end{theorem}

\begin{proof}
Since $ \mathtt{f}(z_0) \neq 0 $ and $ \mathtt{f} $ is continuous, there exists an open disk
$ V_0 \ni z_0 $ on which $ \mathtt{f}(z) \neq 0 $. Shrinking if necessary, choose a
simply connected open subset $ V \subset V_0 $.
On $ V $, the function $ 1/\mathtt{f} $ is holomorphic, so there exists a holomorphic
primitive $\mathscr{F}$ satisfying $ \mathscr{F}'(z) = 1/\mathtt{f}(z) $.
Because $ \mathscr{F}'(z) \neq 0 $, by the holomorphic inverse function theorem,
$ \mathscr{F}: V \to \mathscr{F}(V) $ is a biholomorphism after possibly shrinking $V$.

Now take any solution $ z(t)$ of $ \dot z = \mathtt{f}(z) $ with values in $ V $,
and set $ w(t) := \mathscr{F}(z(t)) $. Then by the chain rule,
\begin{align*}
\dot w(t) = \mathscr{F}'(z(t)) \dot z(t)
           = \frac{1}{\mathtt{f}(z(t))} \mathtt{f}(z(t))
           = 1.
\end{align*}
Hence $ w(t) = t + C $, or equivalently $ \mathscr{F}(z(t)) = t + C $.
Taking the imaginary part gives
\begin{align*}
\mathfrak{Im} \mathscr{F}(z(t)) = \mathfrak{Im} C = \text{constant},
\end{align*}
which proves \ref{111} and \ref{222}.

Because $\mathscr{F}$ is conformal, the change of variable $ w = \mathscr{F}(z) $
transforms the system into the constant vector field $ \dot w = 1 $,
whose integral curves in the $ w $-plane are horizontal lines
$ \{\mathfrak{Im} \ w = \text{constant}\} $.
The inverse image under $ \mathscr{F}^{-1} $ of these horizontal lines
are exactly the level curves of $ \mathfrak{Im} \mathscr{F} $,
which coincide with the trajectories of $ \dot z = \mathtt{f}(z) $ in $ V $.
This proves \ref{333}. Finally, existence and uniqueness of the local flow
follow from the Picard–Lindelöf theorem,
since $\mathtt{f}$ is analytic (hence locally Lipschitz) on $ V $.
Uniqueness ensures that each level curve of $ \mathfrak{Im}\mathscr{F} $
corresponds to a unique trajectory.
\end{proof}

\subsection{Coupling strength $0<\lambda<\omega$} \label{sec: 4.1}

Define $x=x_1-x_2$ and $y=y_1-y_2$. Recall \eqref{eq x+iy} and we write
\begin{equation} \label{n=2}
\left\{ \begin{aligned}
    \dot{x}=& \;\omega-\lambda\sin(x)\cosh(y) := f_1(x,y) ,\\
    \dot{y}=&~~~-\lambda\cos(x)\sinh(y) := f_2(x,y) .
\end{aligned}
\right.
\end{equation}
We pause to remark that it is impossible to arrive at phase-locking synchronization in the real $N=2$ Kuramoto model with a weak coupling ($\lambda<\omega$) because there are no fixed points. However, fixed points exist for the complexified Kuramoto model. 

A direct analysis of the equilibirum points of system~\eqref{n=2} on $\mathbb{R}^2$, which are equivalently the equilibrium points of system~\eqref{n=2 complex} on $\mathbb{C}$, i.e., the zero set $\mathcal{Z}$, reveals that
\begin{align*}
    \mathcal{Z}
    =
    \left\{
    2 k \pi + \frac{\pi}{2} \pm i \cosh^{-1}\left( \frac{\omega}{\lambda} \right)
    :
    k \in \mathbb{Z}
    \right\}
    .
\end{align*}
Recalling \eqref{n=2 complex}, one may observe that $\operatorname{Res}\!\big(1/f,z_k\big)\in i\mathbb{R}$ for all $z_k\in\mathcal{Z}$ (equivalently, every period $ \oint 1/f \ dz$ over a closed loop is real). Therefore, combining Theorem~\ref{local rectification thm} with the “real periods” gluing, if $\widetilde{\mathscr{F}}'=1/f$ then $\mathfrak{Im} \widetilde{\mathscr{F}}$ is single-valued and continuous on $\mathbb{C}\setminus \mathcal{Z}$; equivalently, $\mathfrak{Im}\int_\gamma 1/f \ dz$ is independent of the path $\gamma$. Hence $\widetilde{I}(z):=\mathfrak{Im}\widetilde{\mathscr{F}}(z)$ defines a global conserved quantity on $\mathbb{C}\setminus\mathcal{Z}$, and its level sets $\{\widetilde{I}=\mathrm{const}\}$ coincide with the trajectories of $\dot z=f(z)$. To derive an explicit form of a conserved quantity, let us calculate
\begin{align*}
      \widetilde{\mathscr{F}}(z)=\int_\gamma \frac{\mathrm{d}\zeta}{\omega-\lambda\sin(\zeta)}  
\end{align*}
along the the path $\gamma$ consisted of two segments: from $0$ to $x$ along the real axis, and then from $x+i0=x$ to $x+iy=z$ parallel to the imaginary axis. A straightforward calculation reveals that
\begin{align}
    \mathfrak{Im} \widetilde{\mathscr{F}}(z)
    &=
    \frac{1}{\sqrt{\omega^2-\lambda^2}}
    \log \left(
    \frac{ \omega \cosh(y) - \lambda \sin(x) + \sqrt{\omega^2-\lambda^2} \sinh(y)  }
    {
    \omega \cosh(y) - \lambda \sin(x) - \sqrt{\omega^2-\lambda^2} \sinh(y)
    }
    \right)
    .
\end{align}
Note that for $z=x+iy\notin\mathcal{Z}$, the ratio
\begin{align}
    C(x,y):=
     \frac{ \omega \cosh(y) - \lambda \sin(x) + \sqrt{\omega^2-\lambda^2} \sinh(y)  }
    {
    \omega \cosh(y) - \lambda \sin(x) - \sqrt{\omega^2-\lambda^2} \sinh(y)
    }
\end{align}
is well-defined and positive, as both the nominator and denominator are positive. Hence, by the general argument in Sec.~\ref{sec 4} above for the existence of a conserved quantity, $\mathfrak{Im} \widetilde{\mathscr{F}}(z)$, and hence $C(x,y)$, is a conserved quantity along solution trajectories.

A few observations are in order. First, by local existence and uniqueness together with Lipschitz estimates of $f'(z)$ on an arbitrary horizontal strip $|\mathfrak{Im}~z|\leq B$ with $B>0$, any solution to~\eqref{n=2 complex} with \emph{real} initial conditions remains real for \emph{all} time. Since $C(x,y)=1$ if and only if $y=0$, the conservation of $C$ along any \emph{real} solutions reduces the problem to real Kuramoto systems. Second, for any non-real initial conditions (i.e., $\mathfrak{Im}~z=y\neq 0$), the conservation of $C(x,y)$ along trajectories is equivalent to the conservation of an alternative quantity
\begin{align}
    E(x,y) :=
    \frac{\omega}{\lambda}
    \frac{\cosh(y)}{\sinh(y)} 
    - \frac{\sin(x)}{\sinh(y)}
    ,~\forall~y \neq 0
    ,
\end{align}
which is the ``energy" function proposed in~\cite[Eq.~(5)]{thumler2023synchrony}. In particular, it holds that
\begin{align}
    E(x,y) = \frac{\sqrt{\omega^2-\lambda^2}}{\lambda}
    \frac{ C(x,y) + 1 }{ C(x,y) - 1 }
    ,~\forall~y\neq 0
    .
\end{align}

Before analyzing the stability around equilibria in the following Theorem~\ref{weak N=2}, we briefly address the issue of global existence of solutions in time, since stability is defined through the dynamical behavior as $t\to\infty$. The analysis of level sets of $E$ in~\cite[Supplementary Material]{thumler2023synchrony} and~\cite{thumler2025collective} shows that for all but one admissible value of $E_0:=E(x_0,y_0)$, determined by the initial condition $z_0=x_0+iy_0$, the level set $\{(x,y):E(x,y)=E_0\}$ is bounded in its $y$-components. In particular, from~\cite[Supplementary Material, Eq.~(S11)]{thumler2023synchrony} one can deduce bounds on $|y|$ (dependent on the initial condition) except the case where $|E_0|=\omega/\lambda$ with the level set satisfying
\begin{align}
    |y| = \log \left( \frac{\omega}{\lambda} \csc(x) \right)
    ,
\end{align}
which we show in Appendix~\ref{appen:finite-time_blow-up} can lead to finite-time blow-up solutions (and hence no global existence of solution in time). Since $f'(z)=-\lambda \cos(z)$ is bounded whenever a bound on $|y|=|\mathfrak{Im}~z|$ can be guaranteed, and the conservation of $E$ ensures that any solution stays on one of its level sets, standard global existence and uniqueness theorem of ODEs $\dot{z}=f(z)$ with Lipschitz $f$ applies except when $|E_0|=\omega/\lambda$.

\begin{theorem} \label{weak N=2}
    Let $0<\lambda<\omega$. Then, every equilibrium point $z=x+iy\in\mathcal{Z}$ of the system~\eqref{n=2 complex} or its equivalent system~\eqref{n=2} is Lyapunov stable but not asymptotically stable. 
    
    Moreover, any solution with initial condition $z_0=x_0+iy_0$ satisfying
    \begin{align}
        \label{eq:condition-periodic-orbit}
        \begin{cases}
            &\sin(x_0) > 0\\
            &|y_0| > \log \left( \frac{\omega}{\lambda} \csc(x_0)\right) 
        \end{cases}
    \end{align}
    is a periodic orbit with the common period $T_{\omega,\lambda}=2\pi/\sqrt{\omega^2-\lambda^2}$.
\end{theorem}

\begin{proof} \label{4.1proof}
With the help of conserved quantities of the system like $C(x,y)$, or equivalently, $E(x,y)$, and direct analyses of the geometry of their level sets around each equilibrium point $z=x+iy\in\mathcal{Z}$ such as those conducted in~\cite[Supplementary Material]{thumler2023synchrony}, it follows directly that these equilibrium points are Lyapunov stable but not asymptotically stable. 

Any solution trajectory starting from an initial condition $z_0=x_0+iy_0$ satisfying~\eqref{eq:condition-periodic-orbit} can be shown to satisfy 
\begin{align*}
    |C(x_0,y_0)|>\frac{\omega+\sqrt{\omega^2-\lambda^2}}{\omega-\sqrt{\omega^2-\lambda^2}}>0,
\end{align*}
with the corresponding level set $\mathcal{S}:=\{z=x+iy:C(x,y)=C(x_0,y_0)\}$ being a simple closed curve~\cite[Supplementary Material]{thumler2023synchrony}. Since there is no equilibrium point on this curve, the solution must orbit along it periodically, by the (global) existence of uniqueness of solutions.

To exactly characterize its period, we employ complex analytical tools. By the residue theorem, the period $T_{\omega,\lambda}$ can be calculated as
\begin{equation}
\begin{aligned}
    T_{\omega,\lambda} &= \oint_\gamma \frac{\mathrm{d}\zeta}{\omega-\lambda\sin(\zeta)}=  \mp 2\pi i \operatorname{Res}\!\big(\frac{1}{f},z_k^{\pm}\big)\\
    &=\mp 2\pi i \frac{1}{f'(z_k^{\pm})}= \frac{2\pi}{\sqrt{\omega^2-\lambda^2}},
\end{aligned}
\end{equation}
where $z_k^{\pm} =2 k \pi + \frac{\pi}{2} \pm i \cosh^{-1}\left( \frac{\omega}{\lambda}\right)\in\mathcal{Z}$ and $\gamma: [0,T_{\omega,\lambda}]\rightarrow \mathbb{C}$ is the curve satisfying $\gamma(t)=z(t)\in \mathcal{S}$.
\end{proof}   

\begin{remark}
    In a precursor to this work~\cite{hsiao2023synchronization}, an alternative proof for periodicity of orbits close to each equilibrium point was developed with the intention of not relying the analysis on any conserved quantity. Here, we present a more economical proof based on conserved quantities, and refer any interested reader to~\cite{hsiao2023synchronization} for an alternative. However, it is worth noting that in the present proof, we can further characterize the exact period of these orbits with powerful tools from complex analysis.
\end{remark}

\begin{remark}
When $\lambda = 0$, the system is in a decoupled state,
so each variable evolves independently.
In this case, the dynamics can be written as $\dot{z} = \omega$.
Therefore, for $z$ to travel a distance of $2\pi$, 
the required time is denoted by $T_{\omega,0}$. On the other hand, we notice that for fixed $\omega$, 
$T_{\omega,\lambda}$ tends to infinity as $\lambda \to \omega^-$.
This occurs because the equilibrium point gradually approaches 
the point $\pi/2$ on the real axis, and a homoclinic orbit 
emerges as a result. In the following, we proceed to discuss the case 
$\lambda = \lambda_c$.
\end{remark}

\subsection{Coupling strength $\lambda=\omega$}

When we consider the case $\lambda = \omega$, we observe that a homoclinic orbit may appear.
Without loss of generality, we set $\lambda = 1$. 
Then the system can be written as $\dot{z} = 1 - \sin(z)$.
By separating variables, we obtain
$2\sin\!\left(\tfrac{z}{2}\right)/(\cos\!\left(\tfrac{z}{2}\right) - \sin\!\left(\tfrac{z}{2}\right)) 
    = t + c$,
where $c = a + i b \in \mathbb{C}$ is a constant determined by the initial condition. 

Since finite-time blow-up may occur, we exclude the cases $b = \pm 1$. A straightforward calculation reveals that
$\tan\!\left(\frac{z}{2}\right) = (t + c)/(t + c + 2)$.
Therefore, as $t \to \pm\infty$, we have $\tan(z/2) \to 1$, 
which implies $z \to \pi/2 + 2k\pi$ for $k\in\mathbb{Z}$. 
Hence, except for the finite-time blow-up or degenerate situations corresponding to the real Kuramoto model, 
the trajectory is a homoclinic orbit. Moreover, the conserved quantity can be explicitly computed as $ C(x,y) = \mathfrak{Im} \int_\gamma \frac{1}{1 - \sin(z)} \, dz 
= \frac{\sinh(y)}{-\cosh(y) + \sin(x)}$.
By plotting the contour lines of $C(x,y)$ in Fig~\ref{fig:h orbit}, 
we can clearly observe that all of the level sets pass through the equilibrium point $z = \pi/2$.
\begin{figure}[H]
    \centering
    \subfloat[Contour lines of $C(x,y)$]
        {\includegraphics[width=0.6\textwidth]{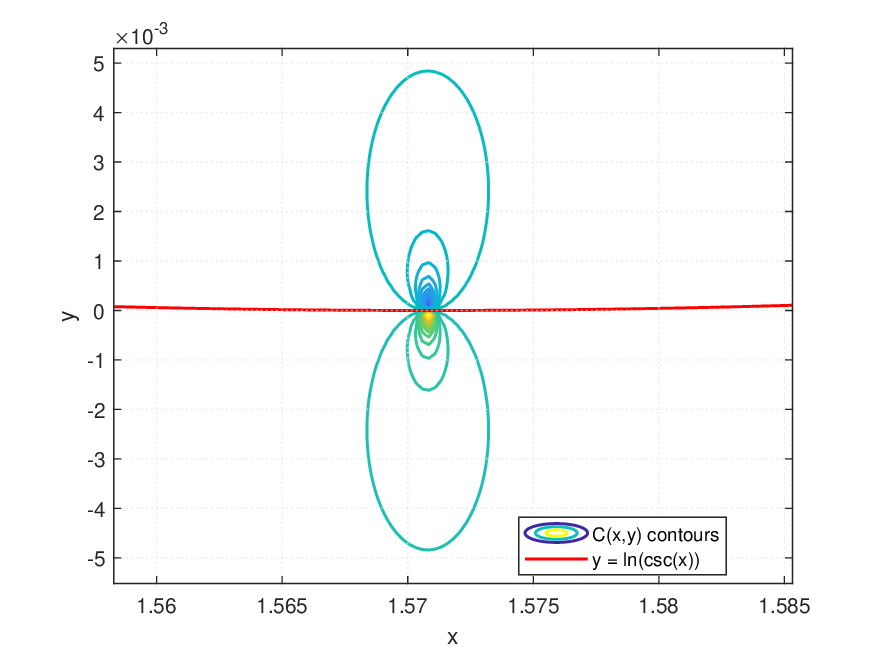}}
    \hfill
    
    \caption{This figure shows the contour lines of the function $C(x,y)$. 
The red curve represents $y = \log(\csc x)$, which corresponds to the level set $C(x,y) = -1$. 
It is known that for any initial condition taken on this curve, the corresponding solution blows up in finite time (see Appendix~\ref{appen:finite-time_blow-up}). Moreover, analytical and numerical computations reveal that the trajectories of the solutions are homoclinic orbits. 
When plotted in the complex plane $\mathbb{C}$, these trajectories coincide with the contour lines of $C$.
}
    \label{fig:h orbit}
\end{figure}

\section{Three oscillators} \label{sec 4.2}
    In this section, we demonstrate that synchronization can be achieved for $N=3$ when the coupling strength satisfies $\lambda<\lambda_c$ by imposing a \emph{Cherry flow}-type ansatz, i.e., we assume the following two conditions in \eqref{main eq} or equivalently \eqref{eq x+iy}. Let $\omega:=\omega_1-\omega_2=\omega_2-\omega_3>0$ and $z_1(0)-z_2(0)=z_2(0)-z_3(0)$. Thus, one may rewrite \eqref{main eq} as
         \begin{equation} \label{n=3} 
    \dot{z}= \;
    \dfrac{\omega}{2}
    -
    \dfrac{\lambda}{3}(\sin(z)+\sin(2z)):=F(z).
\end{equation}
Notice that in this case, $\omega=\lambda_c$ where $\lambda_c$ is defined in \eqref{eq:lambda_c}. We divide the discussion into three cases:
\begin{enumerate}
    \item $ \lambda < \Lambda_c=3\omega/(2\max\limits_{x\in \mathbb{R}}(\sin(x)+\sin(2x)))
    = \sqrt{\frac{69-11\sqrt{33}}{8}} \omega
    = 0.85218915... \omega$. 
    \item $\lambda=\Lambda_c$.
    \item $ \Lambda_c < \lambda < \lambda_c = \omega $.
\end{enumerate}

\subsection{Coupling strength $\lambda<\Lambda_c$} \label{4.2.1} 
In this part, we show that the system of equations \eqref{n=3} possesses two equilibria in each set:
\begin{align}  \label{75 set}
        R_k:=\left\lbrace (x,y) ~\bigg\vert~2k \pi \leq x< 2(k+1) \pi,~y\geq 0\right\rbrace,
\end{align}
where $k\in \mathbb{Z}$, and moreover, one of these equilibria is a stable equilibrium $(x^k,y^k)$ and the other is an unstable equilibrium $(\bar{x}^k,\bar{y}^k)$. We restrict our attention to $y\geq 0$ since it is clear from the system equation~\eqref{n=3} that if $(x(t),y(t))$ is a solution, so is $(x(t),-y(t))$. Due to the periodicity of the system in the real component, without loss of generality, we focus on the equilibria in the set $R_0$. 

Then, by an elementary (but tedious) computation, one may show that the equation $\dot{z}=F(z)$ has exactly two equilibrium points $z_1$ and $z_3$ in $R_0$ with positive imaginary components. Also, one with real component in $(\pi/4,r_1)$ and the other in $(r_3,5\pi/4)$, where $r_1:=\arccos \left( \frac{-1+\sqrt{33}}{8}\right)$ and $r_3 = 2 \pi -\arccos \left( \frac{-1-\sqrt{33}}{8}\right)$. The analysis is derived in Appendix \ref{root of F}.

A straightforward calculation reveals that, for $l=1,3$,
\begin{equation} \label{Re F prime}
\begin{aligned}
    \mathfrak{Re}(F'(z_l))=&-\dfrac{\lambda}{3}\left( \cos(\tilde{x}_l)\cosh(\tilde{y}_l)+2\cos(2\tilde{x}_l)\cosh(2\tilde{y}_l)\right)\\
    \overset{\mathfrak{Im}(F(z))=0}{=}&-
   \frac{\lambda}{3}\cos(x_l)\left( \cosh(y_l)-\dfrac{2\sinh(y_l)\cosh(2y_l)}{\sinh(2y_l)}\right).
\end{aligned}
\end{equation}
Since for all $y>0$, 
\begin{align*} 
&\cosh(y)-\dfrac{2\sinh(y)\cosh(2y)}{\sinh(2y)} <0
,    
\end{align*}
we obtain $z_1$ is an unstable equilibrium and $z_3$ is a stable equilibrium.

\subsection{Coupling strength $\lambda=\Lambda_c$} \label{4.2.2}

We focus on the case $\lambda=\Lambda_c$. We claim that the system of equations \eqref{n=3} possesses two equilibria in $R_k$ (recall \eqref{75 set}). Thus we may verify that one of the equilibrium points is $(x,y)=(r_1,0)$; see Fig.~\ref{fig:M5}. This equilibrium is not stable, since the stable manifold $W^{s}(r_1,0)$ contains at least $\{ (x,y): 0<x \leq r_1,~y=0\}$ and on the other hand, the unstable manifold $W^{u}(r_1,0)$ contains at least $\{ (x,y): r_1<x<2\pi,~y=0\}$, which follows from the real Kuramoto model. That is, we verify the result of Cherry flow with the same natural frequency gaps in \cite{maistrenko2004mechanism} (see Fig.~2(b) therein; our result recovers the stability of equilibria on the diagonal of that figure) when $0.852...<\lambda/\omega<1$, and we characterize the exact lower bound as $3/(2\max\limits_{x\in \mathbb{R}}(\sin(x)+\sin(2x)))
= \sqrt{\frac{69-11\sqrt{33}}{8}
} \approx 0.85218915$, while the rough lower bound on $\lambda/\omega$ given in \cite{maistrenko2004mechanism} corresponds to $1.70/2=0.85$.

Following a similar argument from the case $\lambda<\Lambda_c$ and recalling the definition of $f$, $g_1$ and $g_2$ in \eqref{f function}, \eqref{g1} and \eqref{g2}, we notice that any equilibrium point with $y>0$ cannot exist in either $x\in (\frac{3\pi}{4},r_2) \cup (r_4,\frac{7\pi}{4})$ or $\left(\frac{\pi}{4},r_1\right)$. Therefore, the equilibrium point $z_3=(x_3,y_3)$ with $y_3>0$ and $x_3\in (r_3,5\pi/4)$, is stable.

\begin{figure} \label{fig:M5} 
        \centering 
         \begin{tikzpicture}
            \draw[step=1cm,gray,very thin] (0,-3.14) grid (6.28,3.14);
            \draw[->] (0,0) -- (6.78,0) node[anchor=west] {$x$};
            \draw[->] (0,-3.14) -- (0,3.14);
            \draw[red, line width = 0.4mm]  plot[smooth,domain=0:6.28] (\x,{sin(deg(\x))+sin(deg(2*\x))});
            \draw[red] (3.14,2.24) node[anchor=south] {$h(x):=\sin(x)+\sin(2x)$};
            \draw[blue, line width = 0.4mm]  plot[smooth,domain=-0:6.28] (\x,{cos(deg(\x))+2*cos(deg(2*\x))});
            \draw[blue] (3.14,-2.24) node[anchor=north] {$h'(x)=\cos(x)+2\cos(2x)$};
            \draw[cyan, line width = 0.2mm]  (0,1.625) -- (6.28,1.625) ;
            \draw[cyan,->] (7,1.325) node[anchor=west] { $\Lambda_c<\lambda<\lambda_c$} -- (6.38,1.625) ;
            \draw[cyan, line width = 0.2mm]  (0,1.76) -- (6.28,1.76) ;
            \draw[cyan,->] (7,2.06) node[anchor=west] { $\lambda=\Lambda_c$} -- (6.38,1.76) ;
            \draw[dashed, cyan, line width = 0.2mm]  (0,1.9) -- (6.28,1.9) ;
            \draw[cyan,->] (7,2.795) node[anchor=west] { $\lambda<\Lambda_c$} -- (6.38,2) ;
            \draw[dashed, cyan, line width = 0.2mm] (0,1.5) -- (6.28,1.5) ;
            \draw[cyan,->] (7,0.6) node[anchor=west] {$\frac{3}{2}$} -- (6.38,1.5) ;
            \draw[dashed, black, line width = 0.2mm](0,1.76) node[anchor=east]{$\max\limits_{x\in\mathbb{R}} h(x)$} -- (0.92,1.76) node{$\bullet$} ;
            \draw[dashed, black, line width = 0.2mm] (0.92,1.76) -- (0.92,0) ;
            \draw[dashed, black, line width = 0.2mm] (0.70,1.625) -- (0.70,0) node{$\bullet$} node[anchor=north] {$s_1$} ;
            \draw[dashed, black, line width = 0.2mm] (1.18,1.625) -- (1.18,0) node{$\bullet$} node[anchor=north] {$s_2$} ;
            \draw[black] (0,1/16) -- (0,0) node{$\bullet$} node[anchor=north]{$0$};
            \draw[black] (3.14,1/16) --  (3.14,0) node{$\bullet$} node[anchor=north]{$\pi$};
            \draw[black] (6.28,1/16) --  (6.28,0) node{$\bullet$} node[anchor=north]{$2\pi$};
            \draw[black,dashed] (0.9359,0) node{$\bullet$} -- (0.9359,-3) node{$\bullet$} node[anchor=north]{$r_1$} ;
            \draw[black,dashed] (2.5737,0) node{$\bullet$} -- (2.5737,-3) node{$\bullet$} node[anchor=north]{$r_2$} ;
            \draw[black,dashed] (3.7094,0) node{$\bullet$} -- (3.7094,-3) node{$\bullet$} node[anchor=north]{$r_3$} ;
            \draw[black,dashed] (5.3473,0) node{$\bullet$} -- (5.3473,-3) node{$\bullet$} node[anchor=north]{$r_4$} ;
        \end{tikzpicture} 
         \caption{This plot is an illustration of the function $h(x)$ and its derivative $h'(x)$ on the interval $[0,2\pi]$. We also clarify the values of $3\omega/(2\lambda)$ for $\lambda<\Lambda_c$, $\lambda=\Lambda_c$ and $\Lambda_c<\lambda<\lambda_c$ coupling in cyan: $\lambda<\Lambda_c$ when $3\omega/(2\lambda) > \max\limits_{x\in\mathbb{R}} h(x)$, $\lambda=\Lambda_c$ when $3\omega/(2\lambda)=\max\limits_{x\in\mathbb{R}} h(x)$, and $\Lambda_c<\lambda<\lambda_c$ when $3/2<3\omega/(2\lambda)<\max\limits_{x\in\mathbb{R}} h(x)$.} 
    \end{figure}
\subsection{Coupling strength $\Lambda_c<\lambda<\lambda_c$} \label{Lambda c lambda lambda c}
Finally, we restrict our attention to the case $\Lambda_c<\lambda<\lambda_c$. In this case, the system equations \eqref{n=3} possesses three equilibria in $R_k$ (recall \eqref{75 set}). We find that the first equilibrium point is a stable equilibrium, whose $x$-coordinate belongs to $\left[r_3,\frac{5\pi}{4}\right)$ and $y$-coordinate is strictly positive from \eqref{89}, \eqref{90}. Next, based on Fig.~\ref{fig:M5}, we observe that the second and third equilibria are $(x,y)=(s_1,0)$ and $(x,y)=(s_2,0)$, where $s_1$ is the unique solution to $h(x)=3\omega/(2\lambda)$ with $s_1<r_1$, and $s_2$ is the unique solution to $h(x)=3\omega/(2\lambda)$ with $s_2>r_1$. These two equilibrium points bifurcate from $(r_1,0)$ when $\lambda$ exceeds $\Lambda_c$. 

Recalling \eqref{Re F prime} and Fig.~\ref{fig:M5}, for $l=1,2$, we obtain $\mathfrak{Re}(F'((s_l,0)))=-\lambda h'(s_l)/3.$
Therefore, the equilibrium $(s_1,0)$ is stable and the equilibrium $(s_2,0)$ is unstable, since $h'(s_1)>0$ and $h'(s_2)<0$.

\subsection{Numerical Results} \label{subsubsec:numerical-Cherry}
We conduct brief numerical simulations, the results of which are presented in Fig.~\ref{fig:Cherry}, to verify our theorems in the previous subsections of Sec.~\ref{sec 4.2}. In particular, based on the equivalent ODE system \eqref{n=3} derived from applying the \emph{Cherry flow}-like ansatz in the $N=3$ complexified Kuramoto systems \eqref{eq x+iy}, we numerically evaluate its flow fields (also called slope fields) under different regimes. In all three sub-figures in Fig.~\ref{fig:Cherry}, the horizontal axis represents the $x$ variable, while the vertical axis represents $y$. By periodicity of the system \eqref{n=3} in the $x$ variable, we only plot the flow fields over the horizontal range $x\in[0,2\pi]$; for the purpose of illustration we choose to only plot over the vertical range $y\in[-2,2].$

The system parameters of this numerical evaluation are listed below. We choose the maximum frequency gap $\omega=1$. In Fig.~\ref{fig:Cherry-super-weak}, we choose $\lambda=0.7$, which falls in the coupling regime $\lambda<\Lambda_c$; in Fig.~\ref{fig:Cherry-critically-weak}, we choose $\lambda=\sqrt{(69-11\sqrt{33})/8}=0.85218915...$, which corresponds to the coupling case $\lambda=\Lambda_c$; and in Fig.~\ref{fig:Cherry-weak}, we choose $\lambda=0.99$, which falls in the coupling regime $\Lambda_c<\lambda<\lambda_c$.

The figures are consistent with our theorems in the previous subsections of Sec. \ref{sec 4.2}. First, in Fig.~\ref{fig:Cherry-super-weak} we observe that there are indeed two complex equilibria in $R_0$ (see \eqref{75 set}). With reference to subsection~\ref{4.2.1}, we verify that
the left one is unstable (i.e., $(\bar{x}^0,\bar{y}^0)$) while the right one is stable (i.e., $(x^0,y^0)$). Second, in Fig.~\ref{fig:Cherry-critically-weak} we observe two equilibria in $R_0$, one real and the other complex. With reference to subsection~\ref{4.2.2}, we verify that the left one is unstable (i.e., $(r_1,0)$) while the right one is stable. Finally, in Fig.~\ref{fig:Cherry-weak} we observe three equilibria in $R_0$, two real and the other complex. With reference to subsection~\ref{4.2.2}, we verify that the left real one is stable (i.e., $(s_1,0)$), the right real one is unstable (i.e., $(s_2,0)$), while the complex one is stable. 

\begin{figure}[H]
    \centering
    \subfloat[$\lambda=0.7<\Lambda_c$ . \label{fig:Cherry-super-weak}]
        {\includegraphics[width=0.3\textwidth]{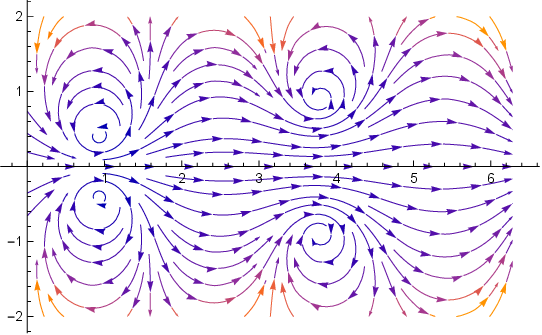}}
    \hfill
    \subfloat[$\lambda=\sqrt{(69-11\sqrt{33})/8}=0.85218915...=\Lambda_c$.\label{fig:Cherry-critically-weak}]
         {\includegraphics[width=0.3\textwidth]{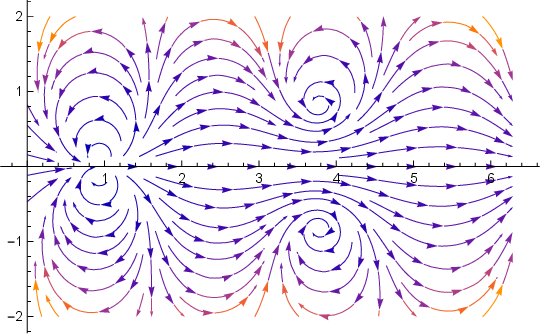}}
    \hfill
    \subfloat[$\Lambda_c<\lambda=0.99<\lambda_c$.\label{fig:Cherry-weak}]
         {\includegraphics[width=0.3\textwidth]{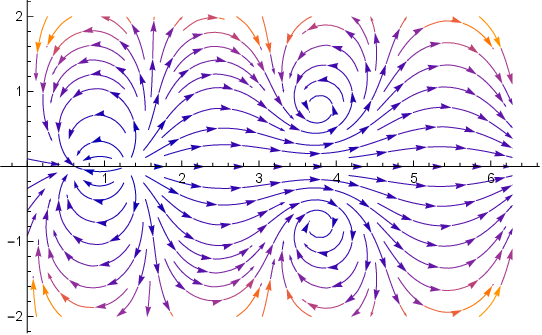}}
    
    \caption{Example flow fields of the equivalent system \eqref{n=3} of $N=3$ oscillators under the Cherry flow-like ansatz. Here, we choose $\omega=1$. The horizontal axis represents the $x$ variable, and the vertical axis represents the $y$ variable.}
    \label{fig:Cherry}
\end{figure}

\section{Acknowledgement}
We thank the two anonymous reviewers and the Associate Editor for their insightful comments and suggestions, which greatly improved the manuscript. We also express our gratitude for the insightful discussions generously provided by Zih-Hao Huang, Yen-Ting Lin, Jun Loo, and Chengbin Zhu.

\section{Data Availability Statement}
Data will be made available on reasonable request.

\section{Declaration}
We do not have any conflict of interest.

\section{Appendix}

\subsection{Solution to $N=2$ complexified Kuramoto model exhibiting finite-time blow-up}  \label{appen:finite-time_blow-up}

We claim in Sec.~\ref{sec1} that for the complexified Kuramoto model~\eqref{eq x+iy} with $N=2$ oscillators, there exists solutions $(x_1(t),y_1(t))$ and $(x_2(t),y_2(t))$ satisfying
\begin{align}
    | y_1(t) - y_2(t) |
    =
    \log \left( \frac{ \omega_1-\omega_2 }{\lambda} \csc( x_1(t) - x_2(t) ) \right)
\end{align}
with finite-time blow-up behavior, where we have assumed $\omega_1>\omega_2$ and $\lambda>0$. We show that this is indeed the case in the following.

Defining $x(t)+iy(t):=z(t):=z_1(t)-z_2(t)$ where $z_j(t):=x_j(t)+iy_j(t)$ for $j=1,2$, and $\omega:=\omega_1-\omega_2>0$ as in Sec.~\ref{sec 4}, we arrive at the same system equation~\eqref{n=2 complex} for $z(t)$ or equivalently~\eqref{n=2} for $(x(t),y(t))$. Since $(x(t),-y(t))$ is also a solution, without loss of generality, we consider $y(t)\geq 0$.

First, we observe that conservation of the ``energy" function
\begin{align}
    E(x,y) := \frac{\omega}{\lambda} \frac{\cosh(y)}{\sinh(y)} - \frac{\sin(x)}{\sinh(y)}
    ,~\forall y\neq 0,
\end{align}
first pointed out in~\cite{thumler2023synchrony} only for the regime of weak coupling $\lambda<\omega$, is in fact a conserved quantity along any non-real solutions for \emph{all} parameter regimes with $\lambda>0$, regardless of the relative strength of $\omega$ versus $\lambda$.

Now, fix any $x_0$ based on different coupling parameter regimes. For $\omega<\lambda$ (strong), take $x_0\in(\pi-\arcsin(\omega/\lambda),\pi)$. For $\omega=\lambda$ (critical), take $x_0\in(\pi/2,\pi)$. For $\omega>\lambda$ (weak), take $x_0\in(0,\pi)$. Then, define
\begin{align}
    y_0 := \log \left( \frac{\omega}{\lambda} \csc(x_0) \right)
    .
\end{align}
A directly computation yields $E(x_0,y_0)=\omega/\lambda$. Observe that
\begin{align}
    E(x,y) = \frac{\omega}{\lambda}
    \Longleftrightarrow
    \lambda \sin(x) = \omega \exp(-y)
    .
\end{align}
Since $E(x,y)$ is conserved along solution trajectories, any solution $(x(t),y(t))$ to~\eqref{n=2} with initial condition $(x_0,y_0)$, extended maximally over time when possible to $[0,T)$ for some $T>0$, must lie on the plane curve defined by $\lambda\sin(x)=\omega\exp(-y)$. In particular, any coordinate $(x,y)$ on this plane curve satisfies $\sin(x)>0$. By our choice of initial condition, this implies that $x(t)<\pi$ for all $t\in[0,T)$.

Recall the $x$-dynamics from~\eqref{n=2}. On the aforementioned plane curve, the $x$-dynamics is equivalently
\begin{align}
    \dot{x}(t)
    &=
    \omega - \lambda \sin(x(t)) \cosh(y(t))
    =
    \omega - \omega e^{-y(t)} \cosh(y(t))
    =
    \omega
    \left(
    \frac{1-e^{-2y(t)}}{2}
    \right)
    .\nonumber
\end{align}
Also, for $t\in[0,T)$, since $\lambda\sin(x(t))=\omega\exp(-y(t))$, we have $ y(t)
    =
    \log\left( \frac{\omega}{\lambda} \csc(x(t))
    \right)$,
and hence the $x$-dynamics also read
\begin{align}
    \dot{x}(t)
    =
    \frac{\omega}{2}
    \left(
    1 - \left(\frac{\lambda}{\omega}\right)^2 \sin^2(x(t))\right)
    .
\end{align}

Next, we aim to show that, in \emph{all} parameter regimes, there exists a positive constant $v_x>0$ such that $\dot{x}(t)\geq v_x$ for all $t\in[0,T)$. When the coupling is weak ($\lambda<\omega$), we can take $v_x = \frac{\omega}{2} \left(
    1 - \left(\frac{\lambda}{\omega}\right)^2 \right)
    > 0$. For critical coupling ($\lambda=\omega$), since $\dot{x}(0)>0$, we can take $ v_x
    =
    \dot{x}(0)
    =
    \frac{\omega}{2}
    \left(
    1 - \sin^2(x_0)\right)
    > 0$.
For strong coupling ($\lambda>\omega$), since $0<\sin(x_0)<\omega/\lambda$ and thus $\dot{x}(0)>0$, we can take $v_x
    =
    \dot{x}(0)
    =
    \frac{\omega}{2}
    \left(
    1 - \left(\frac{\lambda}{\omega}\right)^2 \sin^2(x_0)\right)
    > 0$.

Suppose, for contradiction, that $T=\infty$. Since $\dot{x}(t)\geq v_x>0$ for all $t\in[0,T)$, it only takes a finite amount of time for $x(t)$ to reach $\pi$, in contradiction with our previous result that $x(t)<\pi$ for all $t\in[0,T)$.

\subsection{The roots of $F(z)=0$} \label{root of F}

Let us find the roots of the following system of equations:
\begin{equation} \label{RHS n=3}
\left\{ \begin{aligned}
    \dfrac{\omega}{2}-\dfrac{\lambda}{3}(\sin(x)\cosh(y)+\sin(2x)\cosh(2y))=0,\\
    -\dfrac{\lambda}{3}(\cos(x)\sinh(y)+\cos(2x)\sinh(2y))=0.
\end{aligned}
\right.
\end{equation}
In the following, we first derive a necessary condition for the real part of equilibrium points via a root-finding argument for a polynomial equation in the variable $\sin(2x)$. Then, we develop Lemma~\ref{lemma 4.5}, analogous to root-finding for single variable continuous functions, for bi-variate continuous functions to locate possible equilibrium points of system \eqref{RHS n=3}. We then use the previous necessary condition to rule out invalid candidates.

First, we observe that $y=0$ cannot be a solution of system \eqref{RHS n=3} when the coupling strength satisfies $\lambda<\Lambda_c$. Using hyperbolic function identities, the second equation in system equations \eqref{RHS n=3} yields
\begin{align} \label{77}
    \cos(x)+2\cos(2x)\cosh(y)=0.
\end{align}
By \eqref{77}, we obtain
\begin{align} \label{78}
    \cosh(y)=-\dfrac{\cos(x)}{2\cos(2x)},
\end{align}
which is well-defined, since $\cos(2x)$ cannot be zero; otherwise, via \eqref{77} we have $\cos(x)=0$, which further implies $\cos(2x)=2\cos^2(x)-1=-1$, contradicting $\cos(2x)=0$.

Recalling the hyperbolic function identities for $\cosh(2y)$ and substituting equation system equations \eqref{78} into first equation in \eqref{RHS n=3}, we arrive at
\begin{align} 
    \dfrac{3\omega}{2\lambda}=\sin(x)\left(-\dfrac{\cos(x)}{2\cos(2x)}+2\cos(x)\left(\dfrac{\cos^2(x)}{2\cos^2(2x)}
    -1\right)\right).
\end{align}
We exploit the trigonometric identities and make a straightforward calculation to obtain
\begin{align} \label{80}
    4\sin^3(2x)+\dfrac{6\omega}{\lambda}\sin^2(2x)-3\sin(2x)-\dfrac{6\omega}{\lambda}=0.
\end{align}
Let the cubic polynomial having a root $\sin(2x)$ be
\begin{align}
    p(x)=4x^3+\dfrac{6\omega}{\lambda}x^2-3x-\dfrac{6\omega}{\lambda}.
\end{align}
We observe that $p(-1)=-1<0$ and $p(1)=1>0$, hence there exists at least one real root of $p(x)$ in $(-1,1)$. Also, we have that $p'(-1)=12(1-\omega/\lambda)-3<0$ (since weak coupling, $\lambda>\omega$) and that leading coefficient of $p$ is positive), hence there can only be one real root of $p(x)$ in $(-1,1)$. We can further locate this real root in $(0,1)$ since $p(1)=1>0$ and $p(0)=-6\omega/\lambda<0$. In summary, we have derived a necessary condition for any real part $x$ of an equilibrium point in system \eqref{RHS n=3}, namely, that $\sin(2x)\in(0,1).$


\begin{lemma} [Horizontal cutting lemma] \label{lemma 4.5}
    Consider a continuous function $f:\mathbb{R}^2\rightarrow \mathbb{R}$. Let $g_1: [y_1,y_2]\rightarrow \mathbb{R}$ and $g_2: [x_1,x_2]\rightarrow \mathbb{R}$ be two strictly monotone continuous functions such that $g_1(y_1)=g_2(x_1)$ and $g_1(y_2)=g_2(x_2)$. If $f(x_1,y_1)\cdot f(x_2,y_2)<0$, then the system of equations:
\begin{equation} \label{horizontal cutting lemma}
\left\{ \begin{aligned}
    f(x,y)&=0,\\
    g_2(x)-g_1(y)&=0,
\end{aligned}
\right.
\end{equation}
has at least one root $(\tilde{x},\tilde{y})$, and hence $x_1 < \tilde{x} < x_2$ and $y_1 < \tilde{y} < y_2$.
\end{lemma}

\begin{proof}
Let the continuous function $G$ defined by $G(x)=g^{-1}_1 g_2(x)$. Clearly, because of intermediate value theorem, $G: [x_1,x_2]\rightarrow [y_1,y_2]$ is well-defined. Therefore, we obtain $f(x_1,G(x_1))\cdot f(x_2,G(x_2))<0$. Using the intermediate value theorem again, we complete the proof.
\end{proof}

To apply Lemma~\ref{lemma 4.5} to our system~\ref{RHS n=3}, let
\begin{align} \label{f function}
    f(x,y)
    :=
    \dfrac{3\omega}{2\lambda}
    -(\sin(x)\cosh(y)+\sin(2x)\cosh(2y))
\end{align}
be defined on $(x,y)\in\mathbb{R}^2,$ let
\begin{align} \label{g1}
    g_1(y):=2\cosh(y)
\end{align} 
be defined on $\mathbb{R}^+$ and let
\begin{align} \label{g2}
    g_2(x):=-\dfrac{\cos(x)}{\cos(2x)}
\end{align}
be defined on the domain 
\begin{align} 
    \{x\in[0,2\pi):-\dfrac{\cos(x)}{\cos(2x)}\geq 2\}=
    \left(\dfrac{\pi}{4},r_1\right]
    \bigcup 
    \left(\dfrac{3\pi}{4},r_2\right] \bigcup 
    \left[r_3,\dfrac{5\pi}{4}\right) \bigcup 
    \left[r_4,\dfrac{7\pi}{4}\right),
    \label{intervals}
\end{align}
where $r_l$, $l=1,2,3,4$ satisfies $g_2(r_l)=g_1(0)=2$ and $\pi/4<r_1<3\pi/4<r_2<\pi<r_3<5\pi/4<r_4<7\pi/4$ (see Fig.~\ref{fig:M4}).
\begin{figure} \label{fig:M4} 
        \centering 
        \resizebox{0.9\linewidth}{!}{
        \begin{tikzpicture}
            \draw[step=1cm,gray,very thin] (0,-4) grid (6.28,4);
            \draw[black] (0,1/16) -- (0,0) node{$\bullet$} node[anchor=north]{$0$};
            \draw[black] (3.14,1/16) --  (3.14,0) node{$\bullet$} node[anchor=north]{$\pi$};
            \draw[black] (6.28,1/16) --  (6.28,0) node{$\bullet$} node[anchor=north]{$2\pi$};
            \draw[->] (0,0) -- (6.78,0) node[anchor=west] {$x$};
            \draw[->] (0,-4) -- (0,4);
            \draw[red, line width = 0.4mm]  plot[smooth,domain=0:0.68775] (\x,{-cos(deg(\x))/cos(deg(2*\x)) });
            \draw[dashed, red, line width = 0.2mm] (pi/4,-4) node[anchor=north]{$\frac{\pi}{4}$} -- (pi/4,4) ;
            \draw[red, line width = 0.4mm]  plot[smooth,domain=0.8675:2.275] (\x,{-cos(deg(\x))/cos(deg(2*\x)) });
            \draw[dashed, red, line width = 0.2mm] (3*pi/4,-4) node[anchor=north]{$\frac{3\pi}{4}$} -- (3*pi/4,4) ;
            \draw[red, line width = 0.4mm]  plot[smooth,domain=2.4545:3.83] (\x,{-cos(deg(\x))/cos(deg(2*\x)) });
            \draw[dashed, red, line width = 0.2mm] (5*pi/4,-4) node[anchor=north]{$\frac{5\pi}{4}$} -- (5*pi/4,4) ;
            \draw[red, line width = 0.4mm]
            plot[smooth,domain=4.0082:5.4165] (\x,{-cos(deg(\x))/cos(deg(2*\x)) });
            \draw[dashed, red, line width = 0.2mm] (7*pi/4,-4) node[anchor=north]{$\frac{7\pi}{4}$} -- (7*pi/4,4) ;
            \draw[red, line width = 0.4mm]
            plot[smooth,domain=5.595:6.28] (\x,{-cos(deg(\x))/cos(deg(2*\x)) });
            \draw[red] (3.14,4) node[anchor=south] {$g_2(x):=-\dfrac{\cos(x)}{\cos(2x)}$};
            \draw[step=1cm,gray,very thin] (8,-4) grid (14,4);
            \draw[->] (8,0) -- (14,0) node[anchor=west] {$y$};
            \draw[->] (11,-4) -- (11,4);
            \draw[black] (11,1/16) -- (11,0) node{$\bullet$} node[anchor=north]{$0$};
            \draw[cyan, line width = 0.4mm]
            plot[smooth,domain=11.005:12.32] (\x,{2*cosh(\x-11)});
            \draw[black,-] (11,2);
            \draw[cyan] (11,4.25) node[anchor=south] {$g_1(y):=2\cosh(y)$};
            \draw[blue, line width = 0.4mm] (0,2) node[anchor=east] {$2$} -- (14,2) node[anchor=west] {$2$};
            \draw[blue,dashed] (0.9359,2) node{$\bullet$} -- (0.9359,0) node{$\bullet$} node[anchor=north]{$r_1$} ;
            \draw[blue,dashed] (2.5737,2) node{$\bullet$} -- (2.5737,0) node{$\bullet$} node[anchor=north]{$r_2$} ;
            \draw[blue,dashed] (3.7094,2) node{$\bullet$} -- (3.7094,0) node{$\bullet$} node[anchor=north]{$r_3$} ;
            \draw[blue,dashed] (5.3473,2) node{$\bullet$} -- (5.3473,0) node{$\bullet$} node[anchor=north]{$r_4$} ;
            \draw[black] (11,2+1/16) -- (11,2) node{$\bullet$};
        \end{tikzpicture} }
        \caption{These two plots are illustrations of Lemma~\ref{lemma 4.5} for the specific functions $g_1(y)$ (right) and $g_2(x)$ (left). The indigo line shows that for $l=1,2,3,4,$ $g_2(r_l)=g_1(0)=2$.} 
    \end{figure}
By observing Fig.~\ref{fig:M4}, we can see that for any $y>0$, there are exactly \emph{four} corresponding $x$-values for which $g_1(y)=g_2(x)$ (which is equivalent to the second equation in \eqref{RHS n=3}), one in each interval in \eqref{intervals}. To further locate $r_l$, we observe that $r_l$'s satisfying the equation $g_2(r_l)=2$ can be solved exactly as
\begin{align}
    r_1 &= \arccos 
    \left( 
    \frac{-1+\sqrt{33}}{8}
    \right)
    ,\quad
    r_2 = \arccos 
    \left( 
    \frac{-1-\sqrt{33}}{8}
    \right)
    ,\quad
    \\
    r_3 &= 
    2 \pi -
    \arccos 
    \left( 
    \frac{-1-\sqrt{33}}{8}
    \right)
    ,\quad
    r_4 = 
    2 \pi -
    \arccos 
    \left( 
    \frac{-1+\sqrt{33}}{8}
    \right).
\end{align}

It is straightforward to see that $\sin(2x)<0$ when $x\in (\frac{3\pi}{4},r_2]\bigcup [r_4,\frac{7\pi}{4})$. By the necessary condition $\sin(2x)\in (0,1)$ that we derived for the real part of any equilibrium above, we conclude that no equilibrium can have its real part in these two intervals.

Next, we apply Lemma~\ref{lemma 4.5} on the domains $(\frac{\pi}{4},r_1]$ or $ [r_3,\frac{5\pi}{4})$ respectively (observe via Fig.~\ref{fig:M4} that $g_2(x)$ is monotonic restricted to each) to ensure that there exists equilibria whose real component lie in those intervals. By a direct calculation, we have
\begin{align} \label{87}
    f(r_1,0)&=\dfrac{3\omega}{2\lambda}-\sin(r_1)-\sin(2r_1)
    =
    \dfrac{3\omega}{2\lambda}
    - \max\limits_{x\in\mathbb{R}} (\sin(x)+\sin(2x))
    ,\\ \label{88}
    f\left(\left(\dfrac{\pi}{4}\right)^+,\infty\right)&=-\infty.
\end{align}
The second equality in \eqref{87} can be seen via Fig.~\ref{fig:M5}, noting that $g_2(x)=2$ if and only if $h'(x)=0$, where $h(x):=\sin(x)+\sin(2x)$. From \eqref{87} we have $f(r_1,0)>0$ in the case $\lambda<\Lambda_c$, by definition. Hence by means of \eqref{87}, \eqref{88} and Lemma \ref{lemma 4.5}, we ensure that 
there exits at least one $\tilde{x}_1\in(\pi/4,r_1)$, $\tilde{y}_1>0$ such that $(\tilde{x}_1,\tilde{y}_1)$ is a root for system of equations \eqref{RHS n=3}. From Fig.~\ref{fig:M4} we further conclude that there are no other roots in $R_0$ for which the real part lies in the same interval $(\pi/4,r_1]$.

Applying the same procedure to search the interval $[r_3,5\pi/4)$, we have
\begin{align} \label{89}
    f(r_3,0)&=\dfrac{3\omega}{2\lambda}-\sin(r_3)-\sin(2r_3)
    \approx
    \dfrac{3\omega}{2\lambda}
    - 0.369
    > 0 ,\\ \label{90}
    f\left(\left(\dfrac{5\pi}{4}\right)^-,\infty\right)&=-\infty.
\end{align}
(Notice that via \eqref{89}, even in the case $\lambda<\Lambda_c$ and the case $\Lambda_c<\lambda<\lambda_c$, we still have $f(r_3,0)>0$.) Combining \eqref{89}, \eqref{90} and Lemma \ref{lemma 4.5}, it is immediate that there exists a \emph{unique} pair
$\tilde{x}_3\in(r_3,5\pi/4)$, $\tilde{y}_3>0$ such that $(\tilde{x}_3,\tilde{y}_3)$ is a root for system of system equations \eqref{RHS n=3}.

To recapitulate, we proved that system equations \eqref{n=3} only have two equilibrium points in $R_0$, where one of them has its real component in $(\pi/4,r_1)$ and the other in $(r_3,5\pi/4)$.

\bibliographystyle{siamplain}
\bibliography{references}

\begin{thebibliography}{10}

\bibitem{lasko2007phase}
{\sc L.~Basnarkov and V.~Urumov}, {\em Phase transitions in the {K}uramoto model}, Physical Review E, 76 (2007), p.~057201.

\bibitem{reviewboccaletti2002}
{\sc S.~Boccaletti, J.~Kurths, G.~Osipov, D.~Valladares, and C.~Zhou}, {\em The synchronization of chaotic systems}, Physics Reports, 366 (2002), pp.~1--101.

\bibitem{bottcher2023complex}
{\sc L.~B\"ottcher and M.~A. Porter}, {\em Complex networks with complex weights}, Phys. Rev. E, 109 (2024), p.~024314.

\bibitem{bronski2021}
{\sc J.~C. Bronski, T.~E. Carty, and L.~DeVille}, {\em Synchronisation conditions in the {K}uramoto model and their relationship to seminorms}, Nonlinearity, 34 (2021), p.~5399.

\bibitem{bronski2011}
{\sc J.~C. Bronski, L.~DeVille, and M.~Jip~Park}, {\em {Fully synchronous solutions and the synchronization phase transition for the finite-N {K}uramoto model}}, Chaos: An Interdisciplinary Journal of Nonlinear Science, 22 (2012), p.~033133.

\bibitem{roberto2022}
{\sc R.~C. Budzinski, T.~T. Nguyen, J.~Doàn, J.~Mináč, T.~J. Sejnowski, and L.~E. Muller}, {\em {Geometry unites synchrony, chimeras, and waves in nonlinear oscillator networks}}, Chaos: An Interdisciplinary Journal of Nonlinear Science, 32 (2022), p.~031104.

\bibitem{roberto2023}
{\sc R.~C. Budzinski, T.~T. Nguyen, J.~Doàn, J.~Mináč, T.~J. Sejnowski, and L.~E. Muller}, {\em Analytical prediction of specific spatiotemporal patterns in nonlinear oscillator networks with distance-dependent time delays}, Phys. Rev. Res., 5 (2023), p.~013159.

\bibitem{cestnik2024integrability}
{\sc R.~Cestnik and E.~A. Martens}, {\em Integrability of a globally coupled complex {R}iccati array: Quadratic integrate-and-fire neurons, phase oscillators, and all in between}, Phys. Rev. Lett., 132 (2024), p.~057201.

\bibitem{chen2024phase}
{\sc K.-W. Chen and C.-W. Shih}, {\em Phase-locked solutions of a coupled pair of nonidentical oscillators}, Journal of Nonlinear Science, 34 (2024), p.~14.

\bibitem{chen2022frequency}
{\sc S.-H. Chen, C.-C. Chu, C.-H. Hsia, and S.~Moon}, {\em Frequency synchronization of heterogeneous second-order forced {K}uramoto oscillator networks: A differential inequality approach}, IEEE Transactions on Control of Network Systems, 10 (2023), pp.~530--543.

\bibitem{chen2021synchronization}
{\sc S.-H. Chen, C.-C. Chu, C.-H. Hsia, and M.-C. Shiue}, {\em Synchronization of heterogeneous forced first-order {K}uramoto oscillator networks: A differential inequality approach}, IEEE Transactions on Circuits and Systems I: Regular Papers, 69 (2022), pp.~757--770.

\bibitem{chen2020mathematical}
{\sc S.-H. Chen, C.-H. Hsia, and M.-C. Shiue}, {\em On mathematical analysis of synchronization to bidirectionally coupled {K}uramoto oscillators}, Nonlinear Analysis: Real World Applications, 56 (2020), p.~103169.

\bibitem{choi2014complete}
{\sc S.-H. Choi and S.-Y. Ha}, {\em Complete entrainment of {L}ohe oscillators under attractive and repulsive couplings}, SIAM Journal on Applied Dynamical Systems, 13 (2014), pp.~1417--1441.

\bibitem{choi2012asymptotic}
{\sc Y.-P. Choi, S.-Y. Ha, S.~Jung, and Y.~Kim}, {\em Asymptotic formation and orbital stability of phase-locked states for the {K}uramoto model}, Physica D: Nonlinear Phenomena, 241 (2012), pp.~735--754.

\bibitem{deville2019synchronization}
{\sc L.~DeVille}, {\em {Synchronization and Stability for Quantum {K}uramoto}}, {Journal of Statistical Physics}, 174 (2019), pp.~160--189.

\bibitem{dorfler2011critical}
{\sc F.~D{\"o}rfler and F.~Bullo}, {\em On the critical coupling for {K}uramoto oscillators}, SIAM Journal on Applied Dynamical Systems, 10 (2011), pp.~1070--1099.

\bibitem{dorfler2012synchronization}
{\sc F.~D{\"o}rfler and F.~Bullo}, {\em Synchronization and transient stability in power networks and nonuniform {K}uramoto oscillators}, SIAM Journal on Control and Optimization, 50 (2012), pp.~1616--1642.

\bibitem{dorfler2014synchronization}
{\sc F.~D{\"o}rfler and F.~Bullo}, {\em Synchronization in complex networks of phase oscillators: A survey}, Automatica, 50 (2014), pp.~1539--1564.

\bibitem{ha2021collective}
{\sc S.-Y. Ha, M.~Kang, and B.~Moon}, {\em Collective behaviors of a {W}infree ensemble on an infinite cylinder}, Discrete \& Continuous Dynamical Systems-Series B, 26 (2021).

\bibitem{ha2012class}
{\sc S.-Y. Ha, M.-J. Kang, C.~Lattanzio, and B.~Rubino}, {\em A class of interacting particle systems on the infinite cylinder with flocking phenomena}, Mathematical Models and Methods in Applied Sciences, 22 (2012), p.~1250008.

\bibitem{ha2016collective}
{\sc S.-Y. Ha, D.~Ko, J.~Park, and X.~Zhang}, {\em Collective synchronization of classical and quantum oscillators}, EMS Surv. Math Sci, 3 (2016), pp.~209--267.

\bibitem{ha2018relaxation}
{\sc S.-Y. Ha, D.~Ko, and S.-Y. Ryoo}, {\em On the relaxation dynamics of {L}ohe oscillators on some {R}iemannian manifolds}, Journal of Statistical Physics, 172 (2018), pp.~1427--1478.

\bibitem{ha2020asymptotic}
{\sc S.-Y. Ha and S.~Ryoo}, {\em Asymptotic phase-locking dynamics and critical coupling strength for the {K}uramoto model}, Commun. Math Phys., 377 (2020), pp.~811--857.

\bibitem{hsia2019synchronization}
{\sc C.-H. Hsia, C.-Y. Jung, and B.~Kwon}, {\em On the synchronization theory of {K}uramoto oscillators under the effect of inertia}, Journal of Differential Equations, 267 (2019), pp.~742--775.

\bibitem{hsiao2023synchronization}
{\sc T.-Y. Hsiao, Y.-F. Lo, and W.~Wang}, {\em Synchronization in the quaternionic {K}uramoto model}, arXiv preprint arXiv:2309.01893,  (2023).

\bibitem{hsiao2025equivalence}
{\sc T.-Y. Hsiao, Y.-F. Lo, and C.~Zhu}, {\em On the equivalence of synchronization definitions in the {K}uramoto flow: A unified approach}, arXiv preprint arXiv:2503.19781,  (2025).

\bibitem{huygens1986christiaan}
{\sc C.~Huygens and R.~Blackwell}, {\em Christiaan Huygens' the Pendulum Clock, Or, Geometrical Demonstrations Concerning the Motion of Pendula as Applied to Clocks}, History of Science and Technology Series, Iowa State University Press, 1986.

\bibitem{jadbabaie2004stability}
{\sc A.~Jadbabaie, N.~Motee, and M.~Barahona}, {\em On the stability of the {K}uramoto model of coupled nonlinear oscillators}, in Proceedings of the 2004 American Control Conference, vol.~5, 2004, pp.~4296--4301.

\bibitem{kuramoto1975self}
{\sc Y.~Kuramoto}, {\em Self-entrainment of a population of coupled non-linear oscillators}, in International Symposium on Mathematical Problems in Theoretical Physics, H.~Araki, ed., Berlin, Heidelberg, Germany, 1975, Springer Berlin Heidelberg, pp.~420--422.

\bibitem{bookkuramoto1984}
{\sc Y.~Kuramoto}, {\em Chemical Oscillations, Waves, and Turbulence}, Springer Berlin, Heidelberg, Berlin, Heidelberg, Germany, 1984.

\bibitem{lee2024complexified}
{\sc S.~Lee, L.~Braun, F.~Bönisch, M.~Schröder, M.~Thümler, and M.~Timme}, {\em {Complexified synchrony}}, Chaos: An Interdisciplinary Journal of Nonlinear Science, 34 (2024), p.~053141.

\bibitem{lee2025hopf}
{\sc S.~Lee, L.~J. Kuklinski, M.~Th{\"u}mler, and M.~Timme}, {\em Hopf-induced desynchronization}, Zeitschrift f{\"u}r Naturforschung A,  (2025).

\bibitem{lohe2009nonabelian}
{\sc M.~A. Lohe}, {\em Non-{A}belian {K}uramoto models and synchronization}, Journal of Physics A: Mathematical and Theoretical, 42 (2009), p.~395101.

\bibitem{maistrenko2004mechanism}
{\sc Y.~Maistrenko, O.~Popovych, O.~Burylko, and P.~A. Tass}, {\em Mechanism of desynchronization in the finite-dimensional {K}uramoto model}, Physical Review Letters, 93 (2004), p.~084102.

\bibitem{bookosipov2007}
{\sc G.~Osipov, J.~Kurths, and C.~Zhou}, {\em Synchronization in Oscillatory Networks}, Springer-Verlag Berlin, Berlin, Heidelberg, Germany, 2007.

\bibitem{diego2005thermodynamic}
{\sc D.~Paz\'o}, {\em Thermodynamic limit of the first-order phase transition in the {K}uramoto model}, Physical Review E, 72 (2005), p.~046211.

\bibitem{bookpikovsky2001}
{\sc A.~Pikovsky, M.~Rosenblum, and J.~Kurths}, {\em Synchronization: A Universal Concept in Nonlinear Sciences}, Cambridge Unviersity Press, New York, 2001.

\bibitem{ritchie2018synchronization}
{\sc L.~M. Ritchie, M.~A. Lohe, and A.~G. Williams}, {\em {Synchronization of relativistic particles in the hyperbolic {K}uramoto model}}, Chaos: An Interdisciplinary Journal of Nonlinear Science, 28 (2018), p.~053116.

\bibitem{rodrigues2016kuramoto}
{\sc F.~A. Rodrigues, K.~D.~P. Thomas, P.~Ji, and J.~Kurths}, {\em The {K}uramoto model in complex networks}, Physics Reports, 610 (2016), pp.~1--98.

\bibitem{ryoo2025asymptotic}
{\sc S.-Y. Ryoo}, {\em Asymptotic formation and orbital stability of phase-locked states in {K}uramoto--{L}ohe type synchronization models on lie groups}, Communications in Mathematical Sciences, 23 (2025), pp.~1221--1239.

\bibitem{reviewstrogatz2000}
{\sc S.~Strogatz}, {\em From {K}uramoto to {C}rawford: exploring the onset of synchronization in populations of coupled oscillators}, Physica D: Nonlinear Phenomena, 143 (2000), pp.~1--20.

\bibitem{reviewstrogatz2001}
{\sc S.~Strogatz}, {\em Exploring complex networks}, Nature, 410 (2001), pp.~268--276.

\bibitem{bookstrogatz2004}
{\sc S.~Strogatz}, {\em SYNC: The Emerging Science of Spontaneous Order}, Penguin Adult, New York, NY, 2004.

\bibitem{thumler2025collective}
{\sc M.~Th{\"u}mler}, {\em Collective dynamics of networked, nonlinear, and oscillatory systems},  (2025).

\bibitem{thumler2023synchrony}
{\sc M.~Th{\"u}mler, S.~G. Srinivas, M.~Schr{\"o}der, and M.~Timme}, {\em Synchrony for weak coupling in the complexified {K}uramoto model}, Physical Review Letters, 130 (2023), p.~187201.

\bibitem{van1993lyapunov}
{\sc J.~Van~Hemmen and W.~Wreszinski}, {\em {{L}yapunov function for the {K}uramoto model of nonlinearly coupled oscillators}}, Journal of Statistical Physics, 72 (1993), pp.~145--166.

\bibitem{verwoerd2008global}
{\sc M.~Verwoerd and O.~Mason}, {\em Global phase-locking in finite populations of phase-coupled oscillators}, SIAM Journal on Applied Dynamical Systems, 7 (2008), pp.~134--160.

\bibitem{reviewwinfree1966}
{\sc A.~T. Winfree}, {\em Biological rhythms and the behavior of populations of coupled oscillators}, Journal of Theoretical Biology, 16 (1967), pp.~15--42.

\end{thebibliography}
\end{document}